\documentclass{siamart171218}
\usepackage{amsfonts,amssymb}
\usepackage{graphicx}
\usepackage{subfig}
\usepackage{color}
\usepackage{mathrsfs}

\newcommand\bx{\boldsymbol{x}}
\newcommand\bn{\boldsymbol{n}}
\newcommand\bu{\boldsymbol{u}}
\newcommand\bv{\boldsymbol{v}}

\newcommand\bg{\boldsymbol{g}}

\newcommand\bdf{\tilde{\boldsymbol{f}}}
\newcommand\bbR{\mathbb{R}}
\newcommand\bbN{\mathbb{N}}

\newcommand\dd{\,\mathrm{d}}
\newcommand\Kn{\mathit{Kn}}

\newcommand\mM{\mathcal{M}}
\newcommand\mH{\mathcal{H}}
\newcommand\mF{\mathcal{F}}

\newcommand\bw{\boldsymbol{w}}
\newcommand{\indexk}{\alpha}

\newcommand{\indexi}{\beta}
\newcommand{\indexj}{\gamma}

\newcommand{\mfm}{m}
\newcommand{\lu}{\overline{\bu}}
\newcommand{\lT}{\overline{\theta}}
\newcommand{\tf}{\tilde{f}}
\newcommand{\tQ}{\tilde{Q}}
\newcommand\pd[2]{\dfrac{\partial {#1}}{\partial {#2}}}
\newcommand\pdd[1]{\dfrac{\partial}{\partial {#1}}}

\graphicspath{{images/}}
\newsiamremark{remark}{Remark}

\title{Numerical Simulation of Microflows Using Hermite Spectral
  Methods\thanks{Zhicheng Hu's work is partially supported by the
    National Natural Science Foundation of China (11601229), and the
    Natural Science Foundation of Jiangsu Province of China
    (BK20160784). Zhenning Cai's work was supported by National
    University of Singapore Startup Fund grant
    R-146-000-241-133. Yanli Wang's work was supported by the National
    Natural Science Foundation of China (11501042), and the
    Postdoctoral Science Foundation of China (2018M631233).  The
    computational resources are supported by the highperformance
    computing platform of Peking University, China }}

\author{Zhicheng Hu\thanks{Department of Mathematics, College of Science,
    Nanjing University of Aeronautics and Astronautics, Nanjing
    210016, China (\email{huzhicheng@nuaa.edu.cn}).} \and
  Zhenning Cai\thanks{Department of Mathematics, National
    University of Singapore, Level 4, Block S17, 10 Lower Kent Ridge
    Road, Singapore 119076 (\email{matcz@nus.edu.sg}).} \and
  Yanli Wang\thanks{Department of Engineering, Peking University, Beijing
    100871, China (\email{wang\_yanli@pku.edu.cn}).}}

\begin{document}

\maketitle

\begin{abstract}
  We propose a Hermite spectral method for the spatially inhomogeneous
  Boltzmann equation. For the inverse-power-law model, we generalize
  a class of approximate quadratic collision operators defined in the
  normalized and dimensionless setting to operators for arbitrary
  distribution functions. An efficient algorithm with a fast transform
  is introduced to discretize the new collision operators. The method
  is tested for one- and two-dimensional benchmark microflow problems.
\end{abstract}



\section{Introduction} \label{sec:intro}
Rarefied gas dynamics studies the gas flows when the mean free path of
the gas molecules is comparable to the characteristic length of the
problem we are concerned about. Typical cases include the gas dynamics
in astronautics (large mean free path) and the
micro-electro-mechanical systems (small characteristic length). In
these cases, continuum fluid models such as Euler equations and
Navier-Stokes equations are no longer accurate; on the other hand,
molecular dynamics is still too expensive to solve these problems due
to the huge number of gas molecules. Therefore, people usually adopt
the method in statistical physics to obtain the mesoscopic kinetic
models for rarefied gas dynamics. One of the most important models is
the Boltzmann equation derived from molecular chaos assumption, which
reads
\begin{equation} \label{eq:Boltzmann}
\frac{\partial f}{\partial t} + \nabla_{\bx} \cdot (\bv f)  = Q[f, f],
  \qquad t\in \mathbb{R}^+, \quad \bx \in \Omega \subset \mathbb{R}^3,
    \quad \bv \in \mathbb{R}^3.
\end{equation}
Here the unknown function $f(t,\bx,\bv)$ is the distribution function,
which describes the number density of molecules in the joint
position-velocity ($\bx$-$\bv$) space at time $t$. The right-hand side
of \eqref{eq:Boltzmann} models the collision between gas molecules. It
usually takes the quadratic form:
\begin{equation} \label{eq:quad_col}
  \begin{split}
  &Q[f, f](t,\bx,\bv) = \\
  & \qquad \int_{\mathbb{R}^3} \int_{\bn \perp \bg} \int_0^{\pi}
    [f(t,\bx,\bv_1') f(t,\bx,\bv') - f(t,\bx,\bv_1) f(t,\bx,\bv)]
    B(|\bg|,\chi)
  \dd\chi \dd\bn \dd\bv_1.
  \end{split}
\end{equation}
Here $\bn$ is a unit vector perpendicular to the relative velocity
$\bg = \bv - \bv_1$, and $\bv', \bv_1'$ are post-collisional
velocities
\begin{equation} \label{eq:post_vel}
\begin{aligned}
\bv' &= \cos^2(\chi/2) \bv + \sin^2(\chi/2) \bv_1
  - |\bg| \cos(\chi/2) \sin(\chi/2) \bn, \\
\bv_1' &= \cos^2(\chi/2) \bv_1 + \sin^2(\chi/2) \bv
  + |\bg| \cos(\chi/2) \sin(\chi/2) \bn.
\end{aligned}
\end{equation}
The collision kernel $B(\cdot,\cdot)$ is a non-negative function
describing the interaction between molecules.

It is generally accepted that the Boltzmann equation provides
solutions with enough accuracy in rarefied gas dynamics. The DSMC
(Direct Simulation of Monte Carlo) method \cite{Bird}, a stochastic
numerical solver for the Boltzmann equation, has been widely used in
the simulation. The DSMC method is usually efficient in solving highly
rarefied flows and steady-state problems, whereas in this work, we are
interested in deterministic solvers, which are expected to be better
at flows in the hydrodynamic regime and dynamic problems. Meanwhile,
we also anticipate smoother numerical results and higher order of
convergence using deterministic solvers.

Obviously, the most complicated part of the Boltzmann equation is the
collision term \eqref{eq:quad_col}, which is also supposed to be the
most expensive part in the numerical method. One classical method to
discretize \eqref{eq:quad_col} is the discrete velocity method
\cite{Goldstein1989}, which turns out to be inefficient due to its low
convergence order \cite{Panferov2002}. A much more efficient method is
the Fourier spectral method \cite{Pareschi1996,Mouhot2006,Hu2017}, and
some two-dimensional and three-dimensional simulations have been
carried out based on this method \cite{Dimarco2018}. To seek higher
numerical efficiency, the Hermite spectral method has been introduced
in \cite{Gamba2018, QuadraticCol} to solve spatially homogeneous
Boltzmann equation. The idea of Hermite spectral method can be traced
back to Grad's classical paper \cite{Grad}. Grad's method is based on
the fact that the collision operator vanishes when the distribution
function takes the form of the Maxwellian
$f(\bv) = \rho \mM_{\bu, \theta}(\bv)$, where $\rho, \theta > 0$,
$\bu \in \mathbb{R}^3$ and
\begin{equation}
  \label{eq:general_Maxwellian}
  \mM_{\bu,\theta}(\bv)
  := \frac{1}{\mfm (2\pi \theta)^{3/2}}
  \exp \left( -\frac{|\bv - \bu|^2}{2 \theta} \right),
\end{equation}
where $\mfm$ is the mass of a single gas molecule. Therefore for
smaller Knudsen number, the distribution function is expected to be
closer to the Maxwellian. By such a property, it is natural to
consider the expansion of the distribution function using orthogonal
polynomials with the weight function $\mM_{\bu,\theta}$. These
polynomials are just the Hermite polynomials. Using this expansion,
the explicit formulae for all the equations in Galerkin's method are
derived for the first time in \cite{QuadraticCol}, and for
inverse-power-law models, numerical tests have been carried out.
Meanwhile, a modelling technique is introduced therein to simplify the
collision term so that the computational cost can be reduced.

However, the work in \cite{QuadraticCol} is not ready for the
simulation of the spatially inhomogeneous Boltzmann equation. The
major reason is that the simplified collision model has a simple form
only when the distribution function is represented as ``Grad's
series'', which means that the parameters $\bu$ and $\theta$ in the
weight function $\mM_{\bu,\theta}$ are respectively the local mean
velocity and the local temperature in energy units. When considering
spatially inhomogeneous problems, these parameters vary spatially,
resulting in different basis functions at different spatial locations.
Consequently, the discretization of the spatial derivative becomes
nontrivial.

There are two possible ways to resolve this issue. One is to introduce
projections to deal with operations between distribution functions
represented by different basis functions, which essentially introduces
nonlinearity into the discrete convection term. This approach is used
in \cite{Cai2018} for the linearized collision operator, from which it
can be seen that the implementation is rather involved. The other way
is to use uniform basis functions for all spatial grid points, and
find an appropriate representation for the simplified collision term
proposed in \cite{QuadraticCol}. This paper will follow the second
idea and it will turn out that the implementation is relatively
easier.

Numerical simulations will be done for 1D benchmark problems in
microflows including Couette flows and Fourier flows. For such
problems, numerical results in \cite{Cai2012, Chen2015} have shown
that the BGK-type models cannot provide reliable predictions when the
Knudsen number is large. Results in \cite{Cai2018} show that even for
linearized collision models, obvious deviation can be observed when
compared with DSMC results. In this work, we are going to show better
agreement with DSMC results using our method. Additionally, to test
the efficiency of our method, some preliminary 2D tests for lid-driven
cavity flows are also carried out.

The rest of this paper is organized as follows: Section \ref{sec:pre}
is a review of the background of our method, and the description of
our algorithm is mostly given in Section \ref{sec:ME}, with the
discussion on boundary conditions left to Section \ref{sec:bc}.
Numerical results are exhibited in Section \ref{sec:num}, and the
paper ends with a conclusion in Section \ref{sec:conclusion}.

\section{Preliminaries} \label{sec:pre}
In this section, we are going to provide some preliminary knowledge
for our further discussion, including the discretization of the
distribution function introduced in a previous work \cite{NRxx}, and
the collision kernels we are going to consider later. A brief review
of these topics will be given in the following two subsections.

\subsection{Discretization of the distribution function}
In most cases, when solving the Boltzmann equation, we are not
interested in the distribution function itself. What we are really
concerned about is usually the macroscopic physical quantities such as
the density $\rho$, momentum $\boldsymbol{m}$ and energy $E$. These
quantities are in fact the moments of the distribution function, and
are related to the distribution function $f(t,\bx,\bv)$ by
\begin{equation}
  \label{eq:macro_var}
  \begin{aligned}
  \rho(t,\bx) &= \mfm \int_{\bbR^3} f(t,\bx,\bv) \dd \bv, \\
  \boldsymbol{m}(t,\bx) &= \mfm \int_{\bbR^3}\bv f(t,\bx,\bv) \dd\bv, \\
  E(t,\bx) &= \mfm \int_{\bbR^3} \frac{|\bv|^2}{2} f(t,\bx,\bv) \dd \bv,
  \end{aligned}
\end{equation}
where $\mfm$ is the mass of a single gas molecule. Due to the
importance of the moments, Grad \cite{Grad} proposed an expansion of
the distribution function in the velocity space which has easy access
to these moments, and the approximation of the distribution function
is a truncation of the series. Here we adopt the equivalent notation
used in \cite{NRxx}:
\begin{equation} 
  \label{eq:expansion}
  f(t, \bx, \bv) =
  \sum_{\alpha \in \bbN^3}\tf_{\alpha}^{[\lu, \lT]}(t, \bx)
  \mH_{\alpha}^{[\lu, \lT]}(\bv),
\end{equation}
where $\lu \in \mathbb{R}^3$ and $\lT \in \mathbb{R}_+$ have
respectively the same dimensions as $\bv$ and $|\bv|^2$, and they can
be chosen such that the convergence of the series is fast.\footnote{In
\cite{Grad, NRxx}, the parameters $\lu$ and $\lT$ are chosen to be the
local velocity and scaled temperature of the gas.  Here we are
discussing about a more general form of the expansion.} For any
multi-index $\alpha = (\alpha_1, \alpha_2, \alpha_3) \in
\mathbb{N}^3$, its norm is defined as $|\alpha| = \alpha_1 + \alpha_2
+ \alpha_3$, and the basis function $\mH_{\alpha}^{[\lu, \lT]}(\bv)$
is defined by
\begin{equation}
  \label{eq:mH}
  \mH_{\alpha}^{[\lu, \lT]}(\bv) = \lT^{-\frac{|\alpha|}{2}}H_{\alpha}\left(\frac{\bv -
      \lu}{\sqrt{\lT}}\right)\mM_{\lu, \lT}(\bv),
\end{equation}
with $H_{\alpha}$ being the Hermite polynomial
\begin{equation}
  H_{\alpha}(\bv) =
    (-1)^{|\alpha|} \exp \left( \frac{|\bv|^2}{2} \right)
    \frac{\partial^{|\alpha|}}{\partial v_1^{\alpha_1} \partial
      v_2^{\alpha_2} \partial v_3^{\alpha_3}}
    \left[ \exp \left( -\frac{|\bv|^2}{2} \right) \right].
\end{equation}
An advantage of this expansion is that the coefficients are also
``moments'' of the distribution function. For example, when $\alpha =
0$, the coefficient $\tf_{\alpha}^{[\lu,\lT]}$ is just the density
of the distribution function $\rho$. Other moments can also be easily
represented by these coefficients. More details will be revealed later
in this section.

Our discretization of the distribution function is simply a truncation
of the series \eqref{eq:expansion}:
\begin{equation} 
  \label{eq:expansion_finite}
  f_M(t,\bx,\bv) :=
    \sum_{|\alpha| \leqslant M}\tf_{\alpha}^{[\lu, \lT]}(t, \bx)
  \mH_{\alpha}^{[\lu, \lT]}(\bv) \in \mF_M(\lu, \lT),
  \qquad M \in \bbN,
\end{equation}
where the finite dimensional function space $\mF_M(\lu, \lT)$ is
\begin{equation}
  \label{eq:PM} 
  \mF_M(\lu, \lT) = \mathrm{span} \{ \mH_{\indexk}^{[\lu, \lT]}(\bv)
  \mid |\alpha| \leqslant M\}.
\end{equation}
Apparently, for every $t$ and $\bx$, $f_M$ is an approximation of $f$
in the function space $\mF_M(\lu, \lT)$. For simplicity, from now on,
we will consider $\lu$ and $\lT$ as constants. The coefficients
$\tf_{\alpha}^{[\lu,\lT]}$ will be shortened as $\tf_{\alpha}$, and
the basis function $\mH_{\indexk}^{[\lu, \lT]}$ will be shortened as
$\mH_{\indexk}$. For example, the equation \eqref{eq:expansion_finite}
is simplified as
\begin{equation} \label{eq:f_fM}
  f_M(t,\bx,\bv) :=
    \sum_{|\alpha| \leqslant M}\tf_{\alpha}(t, \bx)
    \mH_{\alpha}(\bv) \in \mF_M(\lu, \lT),
\end{equation}
which looks more concise. However, when we use parameters other than
$\lu$ and $\lT$ in \eqref{eq:expansion}, the parameters will still be
explicitly written out.

One advantage of the approximation \eqref{eq:f_fM} is that the
truncation preserves low-order moments. Precisely speaking, using the
canonical unit vectors $e_1$, $e_2$ and $e_3$ to denote the
multi-indices $(1,0,0)$, $(0,1,0)$ and $(0,0,1)$, we have
\begin{equation} \label{eq:csv_mnts}
  \rho = \tf_0,  \qquad
  \boldsymbol{m} = \rho \lu +
    (\tf_{e_1}, \tf_{e_2}, \tf_{e_3})^T, \qquad
  E = \boldsymbol{m} \cdot \lu - \frac{1}{2} \rho |\lu|^2
    + \frac{3}{2} \rho \lT + \sum_{d=1}^3 \tf_{2e_d},
\end{equation}
which can be derived from the orthogonality of Hermite polynomials
\begin{equation}
  \label{eq:Her_orth}
  \int_{\bbR^3}H_{\alpha}(\bv)H_{\beta}(\bv)
    \exp \left( -\frac{|\bv|^2}{2} \right) \dd \bv
  = \frac{(2\pi)^{3/2}}{\alpha!}
    \delta_{\alpha_1\beta_1}\delta_{\alpha_2\beta_2}\delta_{\alpha_3\beta_3},
\end{equation}
where $\alpha! = \alpha_1! \alpha_2! \alpha_3!$. With these moments,
we can also obtain the mean velocity $\bu$ and the temperature $T$ by
\begin{equation} \label{eq:u_T}
\bu = \frac{\boldsymbol{m}}{\rho}, \qquad
T = \frac{\mfm}{k_B} \frac{2E - \rho |\bu|^2}{3\rho},
\end{equation}
where $k_B$ is the Boltzmann constant. Following the convention, we
define
\begin{equation} \label{eq:theta}
\theta = \frac{k_B}{\mfm} T = \frac{2E - \rho |\bu|^2}{3\rho}.
\end{equation}
It can be seen that all the quantities from \eqref{eq:csv_mnts} to
\eqref{eq:theta} are not changed by the truncation \eqref{eq:f_fM} if
$M \geqslant 2$. More generally, by the orthogonality
\eqref{eq:Her_orth}, we can obtain the coefficients $\tf_{\alpha}$
from the distribution function by
\begin{equation}
  \label{eq:coe}
  \tf_{\alpha}(t, \bx) =
    \frac{\mfm}{\alpha!} \lT^{\frac{|\alpha|}{2}}
  \int_{\bbR^3}H_{\alpha}\left(\frac{\bv -
      \lu}{\sqrt{\lT}}\right)f(t, \bx, \bv)\dd \bv,
  \qquad |\alpha| \leqslant M.
\end{equation}
With \eqref{eq:coe}, other interesting moments such as the stress
tensor $\sigma_{ij}$ and the heat flux $q_i$, which are defined by
\begin{align*}
  \sigma_{ij} &= \mfm\int_{\bbR^3} \left( (v_i-u_i)(v_j -u_j) - \frac{1}{3}
    \delta_{ij} |\bv-\bu|^2 \right) f \dd \bv, \quad i,j = 1,2,3, \\
  q_i &= \frac{1}{2}\mfm\int_{\bbR^3}|\bv -\bu|^2(v_i-u_i)f\dd \bv, \qquad i = 1, 2, 3,
\end{align*}
can also be easily related to the first few coefficients as follows:
\begin{align*}
  \sigma_{ij} &= (1+\delta_{ij})\tf_{e_i+e_j}
    +\delta_{ij} \rho(\lT - \theta)
    -\rho(\overline{u}_i - u_i)(\overline{u}_j -u_j), \\
  q_i &= 2\tf_{3e_i}
    + (\overline{u}_i - u_i)\tf_{2e_i}
    + |\lu - \bu|^2\tf_{e_i} \\
  & \quad + \sum_{k=1}^3\left[\tf_{2e_k+e_i}
    + (\overline{u}_k - u_k)\tf_{e_k+e_i}
    + (\overline{u}_i - u_i)\tf_{2e_k} \right].
\end{align*}
The above expressions involve only the parameters $\lu$, $\lT$ and the
coefficients with index norm less than or equal to $3$.

In our numerical method, $f_M(t,\bx,\bv)$ defined in \eqref{eq:f_fM}
will be used as the semi-discrete distribution function. The major
difficulty in the discretization of the equation lies in the collision
operator, which will be discussed in detail in the following sections.

\subsection{Collision kernels}
In order to apply \eqref{eq:f_fM} to the numerical scheme, we first
need to define the collision kernel $B(\cdot, \cdot)$ (see
\eqref{eq:quad_col}). In this paper, we are interested in the
inverse-power-law (IPL) model, for which
\begin{equation}
  \label{eq:IPL}
  B(|\bg|, \chi)
    =\left(\frac{2\kappa}{\mfm}\right)^{\frac{2}{\eta-1}} 
    |\bg|^{\frac{\eta - 5}{\eta-1}} W_0
    \left| \frac{\mathrm{d} W_0}{\mathrm{d} \chi} \right|,
  \quad \eta > 3.
\end{equation}
Here $\eta$ is an index indicating the decay rate of the repulsive
force between gas molecules when their distance increases, and
$\kappa$ is a constant indicating the intensity of the potential. The
angle $\chi$ and dimensionless impact parameter $W_0$ are related by
\begin{equation} \label{eq:chi}
\chi = \pi - 2 \int_0^{W_1} \left[
  1 - W^2 - \frac{2}{\eta - 1} \left( \frac{W}{W_0} \right)^{\eta-1}
\right]^{-1/2} \dd W,
\end{equation}
with $W_1$ being a positive real number satisfying
\begin{equation} \label{eq:W_1}
1-W_1^2 - \frac{2}{\eta - 1} \left( \frac{W_1}{W_0} \right)^{\eta-1} = 0.
\end{equation}
Details about this model can be found in \cite{Bird}. The
inverse-power-law model works well for a wide range of gases around
the room temperature. We refer the readers to \cite{Cowling} for more
details.

Another related model is the variable-hard-sphere (VHS) model, which
is proposed by Bird in \cite{Bird1963} as an approximation of the IPL
model. The collision kernel of the VHS model is
\begin{equation}
  \label{eq:VHS}
  B(|\bg|, \chi) = \frac{1}{4}d_{\rm ref}^2 g_{\rm ref}^{2\nu}
    |\bg|^{1 - 2\nu}\sin\chi, \qquad 0 < \nu \leqslant 1,
\end{equation}
where $d_{\rm ref}$ is the reference molecular diameter and $g_{\rm
ref}$ is the reference speed. When approximating the IPL model with
index $\eta$, the parameter $\nu$ is chosen as $2/(\eta-1)$.



\section{Hermite spectral method for the Boltzmann equation}
\label{sec:ME}
Now we are ready to find the evolution equations for the coefficients
$\tf_{\alpha}$ in \eqref{eq:f_fM}, which contains the discretization
of the convection term and the collision term of the Boltzmann
equation \eqref{eq:Boltzmann}. Based on the idea of Galerkin's method,
we need to expand both the convection term and the collision term with
the same basis functions $\mH_{\alpha}$ as in \eqref{eq:f_fM}. The
major difficulty lies in the collision term, which will be discussed
first below.

\subsection{Series expansions of the IPL collision terms}
\label{sec:general}
In what follows, we will study the series expansion of the quadratic
collision term $Q[f, f]$ defined in \eqref{eq:quad_col}. Precisely
speaking, for any given $\lu$ and $\lT$, the binary collision term
$Q[f, f]$ is to be expanded as
\begin{equation}
  \label{eq:S_expan}
  Q[f, f](t, \bx, \bv) = \sum_{\alpha \in \bbN^3}
  \tQ_{\alpha}(t, \bx)\mH_{\alpha}(\bv).
\end{equation}
By the orthogonality of Hermite polynomials \eqref{eq:Her_orth},
the coefficients can be evaluated by
\begin{equation}
\label{eq:S_k1k2k3}
\tQ_{\indexk}(t, \bx)
= \frac{\mfm}{\alpha!}\lT^{\frac{|\alpha|}{2}}\int_{\bbR^3}
H_{\alpha}\left(\frac{\bv - \lu}{\sqrt{\lT}}\right)Q[f,
f](t,\bx,\bv) \dd \bv.
\end{equation}
Since we are focusing on the collision operator, below in Sections
\ref{sec:general} and \ref{sec:approx}, the variables $t$ and $\bx$
will be temporarily omitted.

In \cite{QuadraticCol}, the authors have proposed an algorithm to find
the values of these coefficients for the dimensionless Boltzmann
collision operator in a special case $\lu = 0$ and $\lT = 1$. Thus, in
order to make use of the result in \cite{QuadraticCol}, we will first
apply the nondimensionalization by defining $\hat{\bv}$ and
$h(\hat{\bv})$ as
\begin{equation}
  \label{eq:changevar}
  \bv = \lu + \sqrt{\lT}{\hat{\bv}}, \qquad
  h(\hat{\bv}) = \frac{\mfm \lT^{3/2}}{\rho}
    f(\lu + \sqrt{\lT} \hat{\bv}),
\end{equation}
where $\rho$ is the density defined in \eqref{eq:macro_var}. From
\eqref{eq:f_fM}, one can derive the series expansion of $h(\hat{\bv})$
as
\begin{equation}
  \label{eq:expansion_sf}
  h(\hat{\bv}) =
  \sum_{\alpha \in \bbN^3}\tilde{h}_{\alpha} \hat{\mH}_{\alpha}(\hat{\bv}),
\end{equation}
where
$\hat{\mH}_{\alpha}(\hat{\bv}) = \mfm\mH_{\alpha}^{[0, 1]}(\hat{\bv})$
is the dimensionless basis function (see \eqref{eq:mH} for the
definition of $\mH_{\alpha}^{[0,1]}$), and the coefficients
\begin{equation}
  \label{eq:tilde_f}
  \tilde{h}_{\alpha} =
    \rho^{-1} \lT^{-\frac{|\alpha|}{2}}\tf_{\alpha},
  \qquad \alpha \in \bbN^3
\end{equation}
are also dimensionless. Using the above definitions, the collision
term $Q[f,f]$ \eqref{eq:quad_col} changes to
\begin{equation}
  \label{eq:scale_Q}
  \begin{split}
    &Q[f, f](\lu + \sqrt{\lT}\hat{\bv}) =  \\
    & \qquad \frac{\rho^2}{\mfm^2\lT^{\frac{3}{2}}}
    \int_{\mathbb{R}^3} \int_{\bn \perp \hat{\bg}}
    \int_0^{\pi} B\left(\sqrt{\lT}|\hat{\bg}|,\chi\right)
    \left[h(\hat{\bv}_1') h(\hat{\bv}') -
      h(\hat{\bv}_1)h(
      \hat{\bv})\right] \dd\chi \dd\bn \dd\hat{\bv}_1,
\end{split}
\end{equation}
where $\hat{\bg} = \hat{\bv} - \hat{\bv}_1$. Specifically, for the
IPL model, it is convenient to define the dimensionless collision
kernel $\hat{B}$ by
\begin{equation}
  \hat{B}(|\hat{\bg}|, \chi) =
    |\hat{\bg}|^{\frac{\eta-5}{\eta-1}} W_0
    \left| \frac{\mathrm{d}W_0}{\mathrm{d}\chi} \right|,
\end{equation} 
and then the IPL collision term turns out to be
\begin{equation}
  \label{eq:scale_Q_IPL}
  Q[f,f](\lu + \sqrt{\lT}\hat{\bv}) =
  \frac{\rho^2}{\mfm^2\lT^{\frac{3}{2}}}
  \left(\frac{2\kappa}{\mfm}\right)^{\frac{2}{\eta-1}}\lT^{\frac{\eta-5}{2(\eta-1)}}
  \hat{Q}[h, h](\hat{\bv}),
\end{equation}
where $\hat{Q}[h, h]$ is the dimensionless collision operator
\begin{equation}
  \label{eq:scale_IPL}
  \hat{Q}[h, h](\hat{\bv}) = \int_{\mathbb{R}^3} \int_{\bn \perp \hat{\bg}}
  \int_0^{\pi} \hat{B}\left(|\hat{\bg}|,\chi\right)
  \left[h(\hat{\bv}_1') h(\hat{\bv}') -
    h(\hat{\bv}_1)h(
    \hat{\bv})\right] \dd\chi \dd\bn \dd\hat{\bv}_1.
\end{equation}
Inserting \eqref{eq:scale_Q_IPL} into \eqref{eq:S_k1k2k3}, we obtain
\begin{equation}
  \label{eq:IPL_Qk}
  \tQ_{\alpha}
  = \frac{\rho^2}{\mfm} 
  \left(\frac{2\kappa}{\mfm}\right)^{\frac{2}{\eta-1}}\lT^{\frac{\eta-5}{2(\eta-1)}
  + \frac{|\alpha|}{2}}
  \frac{1}{\alpha!}\int_{\bbR^3} H_{\alpha}(\hat{\bv})\hat{Q}[h, h](\hat{\bv})
  \dd \hat{\bv}, \qquad \alpha \in \bbN^3.
\end{equation}
For the IPL model, the integral in \eqref{eq:IPL_Qk} has been deeply
studied in \cite{QuadraticCol}. The general result is
\begin{equation}
  \label{eq:expansion_Q}
\frac{1}{\alpha!}\int_{\bbR^3} H_{\alpha}(\hat{\bv})\hat{Q}[h, h](\hat{\bv})
  \dd \hat{\bv} = \sum_{\beta \in\bbN^3} \sum_{\gamma \in \bbN^3}
  A_{\alpha}^{\beta, \gamma} \tilde{h}_{\beta}\tilde{h}_{\gamma},
  \qquad \alpha \in \bbN^3.
\end{equation}
The coefficients $A_{\alpha}^{\beta, \gamma}$ are constants for a
given collision model, and an algorithm to compute these coefficients
is given \cite{QuadraticCol} for all IPL models. The algorithm
uses an explicit expression of $A_{\alpha}^{\beta,\gamma}$ which
involves only a one-dimensional integral, and this integral is
evaluated by adaptive numerical integration. Such an algorithm can
provide very accurate values for these coefficients, and has been
verified to be reliable in the computation of homogeneous Boltzmann
equation. Due to the lengthy expressions involved in the algorithm, we
are not going to repeat the details in this paper. We would just like
to mention that in order to find the values of all
$A_{\alpha}^{\beta,\gamma}$ with $|\alpha|,|\beta|,|\gamma| \leqslant
M$, the computational cost is proportional to $M^{12}$. Although the
time complexity is high, all these coefficients can be pre-computed
and stored. Readers are referred to \cite{QuadraticCol} for the
details of the algorithm.

Substituting \eqref{eq:tilde_f} and \eqref{eq:expansion_Q} into
\eqref{eq:IPL_Qk}, we finally get
\begin{equation}
  \label{eq:expansion_IPL}
  \tQ_{\alpha} = \frac{1}{\mfm} 
  \left(\frac{2\kappa}{\mfm}\right)^{\frac{2}{\eta-1}}
  \lT^{\frac{\eta-5}{2(\eta-1)}} \sum_{\indexi \in \bbN^3}
  \sum_{\indexj \in \bbN^3} \lT^{\frac{1}{2}(|\alpha| - |\beta|
    -|\gamma|)}A_{\alpha}^{\beta, \gamma}\tf_{\beta}\tf_{\gamma},
  \qquad \alpha \in \bbN^3.
\end{equation}
As mentioned at the beginning of Section \ref{sec:general}, the
parameters $\lu$ and $\lT$, which affects the coefficients
$\tf_{\beta}$ and $\tf_{\gamma}$ implicitly, can be arbitrarily
chosen. A special choice is $\lu = \bu$ and $\lT = \theta$ (see
\eqref{eq:macro_var}\eqref{eq:u_T} and \eqref{eq:theta} for the
definitions), which leads to
\begin{equation}
  \label{eq:expansion_IPL_std}
  \tQ^{[\bu,\theta]}_{\alpha} = 
  \frac{c \theta}{\mu}\sum_{\beta \in \bbN^3}
  \sum_{\gamma \in \bbN^3} \theta^{\frac{1}{2}(|\alpha| - |\beta|
    -|\gamma|)}A_{\alpha}^{\beta, \gamma}
  \tf^{[\bu,\theta]}_{\beta}\tf^{[\bu,\theta]}_{\gamma},
  \qquad \alpha \in \bbN^3,
\end{equation}
where $\mu$ is the viscosity coefficient (see \cite[eq. (3.62)]{Bird})
\begin{equation}
\mu = \frac{5\mfm (\theta/\pi)^{1/2}(2\mfm\theta/\kappa)^{2/(\eta-1)}}
  {8 A_2(\eta) \Gamma(4-2(\eta-1))}, \qquad
A_2(\eta) = \int_0^{+\infty} W_0 \sin^2 \chi \dd W_0,
\end{equation}
and $c$ is a constant given by%
\footnote{When $\lu = \bu$ and $\lT = \theta$, the expansion
\eqref{eq:expansion} is identical to the one proposed by Grad in
\cite{Grad}, where the expansion of the collision term is also
considered. For example, the equation (A3.56) is
\begin{displaymath}
J_{ij}^{(2)} = -\frac{6}{m} B_1^{(2)} \rho a_{ij}^{(2)} + \cdots.
\end{displaymath}
When $i = 1$ and $j = 2$, it can be translated to our language:
\begin{displaymath}
\tQ_{\varsigma}^{[\bu,\theta]} =
  -\frac{6}{m} B_1^{(2)} \rho \tf_{\varsigma}^{[\bu,\theta]} + \cdots,
\end{displaymath}
by using $\tQ_{\varsigma}^{[\bu,\theta]} = \rho \theta J_{12}^{(2)}$
and $\tf_{\varsigma}^{[\bu,\theta]} = \rho \theta a_{12}^{(2)}$.
Comparing this equation with \eqref{eq:expansion_IPL}, we find
\begin{displaymath}
\frac{1}{m} \left( \frac{2\kappa}{m} \right)^{\frac{2}{\eta-1}}
  \theta^{\frac{\eta-5}{2(\eta-1)}}
  (A_{\varsigma}^{0,\varsigma} + A_{\varsigma}^{\varsigma,0}) =
-\frac{6}{m} B_1^{(2)}.
\end{displaymath}
By $\mu = \frac{m\theta}{6B_1^{(2)}}$ (the equation (5.30) in
\cite{Grad}), we obtain the coefficients in front of the sums in
\eqref{eq:expansion_IPL_std}.}%
\begin{equation}
c = -\left( A_{\varsigma}^{0,\varsigma} + A_{\varsigma}^{\varsigma,0}
  \right)^{-1}, \qquad \varsigma = (1,1,0).
\end{equation}
Such a special case will be used in the next section when we reduce
the computational cost by simplifying the collision term.

\subsection{Approximation to the Boltzmann collision term}
\label{sec:approx}
The previous section establishes the basic theory for discretization
of the collision term. By Galerkin's method, the equations for
$\tf_{\alpha}$ should hold the form
\begin{equation}
\pd{\tf_{\alpha}}{t} + \cdots = \tQ_{\alpha},
  \qquad |\alpha| \leqslant M,
\end{equation}
where $\cdots$ denotes the corresponding convection term to be
discussed in Section \ref{sec:eqs}, and $\tQ_{\alpha}$ is given in
\eqref{eq:S_k1k2k3} with $f$ replaced by $f_M$ (see \eqref{eq:f_fM}).
By \eqref{eq:expansion_IPL}, it is known that the total computational
cost for evaluating all $\tQ_{\alpha}$ with $|\alpha| \leqslant M$ is
$O(M^9)$, which is unacceptable for a large $M$, especially for
spatially inhomogeneous problems. The aim of this section is to build 
new collision models and derive the corresponding $\tQ_{\alpha}$
following the method proposed in \cite{QuadraticCol}.

A strategy to reduce the computational cost has been proposed in
\cite{QuadraticCol} based on the dimensionless and normalized
settings. To demonstrate the result, we consider again the
dimensionless distribution function $h(\hat{\bv})$ defined in the
previous section. When $h(\hat{\bv})$ satisfies
\begin{equation} \label{eq:h}
\int_{\bbR^3} \hat{\bv} h(\hat{\bv}) \dd \hat{\bv} = 0, \qquad
\frac{1}{3} \int_{\bbR^3} |\hat{\bv}|^2 h(\hat{\bv}) \dd \hat{\bv} = 1,
\end{equation}
the dimensionless collision operator $\hat{Q}[h,h]$ is approximated by
\begin{equation} \label{eq:Qstar}
\hat{Q}^*[h,h](\hat{\bv}) = \sum_{|\alpha| \leqslant M_0}
  \sum_{|\beta| \leqslant M_0} \sum_{|\gamma| \leqslant M_0}
    A_{\alpha}^{\beta,\gamma} \tilde{h}_{\beta} \tilde{h}_{\gamma}
      \hat{\mH}_{\alpha}(\hat{\bv}) -
  \sum_{|\alpha| > M_0} \nu_{M_0} \tilde{h}_{\alpha}
    \hat{\mH}_{\alpha}(\hat{\bv}),
\end{equation}
where $M_0$ is an arbitrarily chosen positive integer, and $\nu_{M_0}$
gives the decay rate of the higher-order coefficients. The idea is to
apply the quadratic collision operator only to the first few
coefficients, and for the remaining coefficients, we adopt the idea of
the BGK-type operator and simply let it decay to zero exponentially at
a constant rate. Unfortunately, the conditions \eqref{eq:h} do not
hold in general. By \eqref{eq:changevar}, it can be found that only in
the special case $\lu = \bu$, $\lT = \theta$, the equalities
\eqref{eq:h} are true, and thus we can recover the dimensions from
\eqref{eq:Qstar} and get the following approximation of $Q[f,f]$:
\begin{equation} \label{eq:Qstar_f}
\begin{split}
Q^*[f,f](\bv) &= \frac{c\theta}{\mu} \sum_{|\alpha| \leqslant M_0}
  \sum_{|\beta| \leqslant M_0} \sum_{|\gamma| \leqslant M_0}
    \theta^{\frac{1}{2}(|\alpha| - |\beta| - |\gamma|)}
    A_{\alpha}^{\beta,\gamma}
    \tf^{[\bu,\theta]}_{\beta} \tf^{[\bu,\theta]}_{\gamma}
      \mH_{\alpha}^{[\bu,\theta]}(\bv) \\
&- \frac{c\theta}{\mu} \sum_{|\alpha| > M_0} \nu_{M_0}
    \rho \tf^{[\bu,\theta]}_{\alpha}
    \mH^{[\bu,\theta]}_{\alpha}(\bv).
\end{split}
\end{equation}
When $\lu = \bu$ or $\lT = \theta$ does not hold, we cannot use the
same way to construct the approximate collision operators, i.e. we
cannot just remove the superscripts $[\bu,\theta]$ in
\eqref{eq:Qstar_f}, since the resulting operator has no guarantee that
it will vanish for Maxwellians, which is a fundamental property of the
Boltzmann equation.

To overcome such a difficulty, we choose to evaluate the coefficients
$\tf^{[\bu,\theta]}_{\alpha}$ based on the knowledge of all the
coefficients $\tf_{\alpha}$, and then \eqref{eq:Qstar_f} can be
applied. The algorithm to obtain $\tf^{[\bu,\theta]}_{\alpha}$ is
inspired by the method proposed in \cite{Qiao}, and can be stated
by the following theorem:
\begin{theorem} \label{thm:projection}
Suppose the function $\phi(\bv)$ satisfies
\begin{equation} \label{eq:fin}
\int_{\bbR^3} (1 + |\bv|^M) |\phi(\bv)| \dd \bv < +\infty,
\end{equation}
for some positive integer $M$. Given $\bw, \bw^* \in \bbR^3$ and $\eta,
\eta^* > 0$, for any $\alpha \in \bbN^3$ satisfying $|\alpha|
\leqslant M$, define
\begin{equation}
\tilde{\phi}_{\alpha} = \frac{1}{\alpha!} \eta^{\frac{|\alpha|}{2}}
  \int_{\bbR^3} H_{\alpha} \left(
    \frac{\bv - \bw}{\sqrt{\eta}}
  \right) \phi(\bv) \dd \bv, \quad
\tilde{\phi}^*_{\alpha} = \frac{1}{\alpha!} (\eta^*)^{\frac{|\alpha|}{2}}
  \int_{\bbR^3} H_{\alpha} \left(
    \frac{\bv - \bw^*}{\sqrt{\eta^*}}
  \right) \phi(\bv) \dd \bv.
\end{equation}
Then
\begin{equation} \label{eq:phi_star}
\tilde{\phi}^*_{\alpha} =
  \sum_{k=0}^{|\alpha|} \tilde{\phi}_{\alpha}^{(k)},
\end{equation}
where $\tilde{\phi}_{\alpha}^{(k)}$ is recursively defined by
\begin{equation}
  \label{eq:sol_odes1}
  \tilde{\phi}_{\alpha}^{(k)} = \left\{ \begin{array}{ll}
    \tilde{\phi}_{\alpha}, & \text{if } k = 0, \\[5pt]
    \displaystyle \frac{1}{k}\sum_{j=1}^3\left( (w_d^* - w_d)
      \tilde{\phi}_{\alpha-e_j}^{(k-1)} + \frac{1}{2} (\eta^* - \eta)
      \tilde{\phi}_{\alpha-2e_j}^{(k-1)}\right),
    & \text{if } 1 \leqslant k \leqslant |\alpha|.
  \end{array} \right.
\end{equation}
Here the terms with negative values in the subscript indices are
regarded as zero.
\end{theorem}

\begin{proof}
For $\tau \in [0,1]$, define the functions
\begin{equation}
\bw(\tau) = (1-\tau) \bw + \tau \bw^*, \qquad
\eta(\tau) = (1-\tau) \eta + \tau \eta^*,
\end{equation}
and
\begin{equation} \label{eq:phi_tau}
\tilde{\phi}_{\alpha}(\tau) = \frac{1}{\alpha!}
  [\eta(\tau)]^{\frac{|\alpha|}{2}}
  \int_{\bbR^3} H_{\alpha} \left(
    \frac{\bv - \bw(\tau)}{\sqrt{\eta(\tau)}}
  \right) \phi(\bv) \dd \bv, \qquad |\alpha| \leqslant M.
\end{equation}
The condition \eqref{eq:fin} ensures that
$\tilde{\phi}_{\alpha}(\tau)$ exists for any $\tau \in [0,1]$.
Especially, we have $\tilde{\phi}_{\alpha}(0) = \tilde{\phi}_{\alpha}$
and $\tilde{\phi}_{\alpha}(1) = \tilde{\phi}^*_{\alpha}$. Now we take
the derivative of \eqref{eq:phi_tau} with respect to $\tau$. By
straightforward calculation, we obtain
\begin{equation} \label{eq:odes}
  \frac{\mathrm{d} \tilde{\phi}_{\alpha}(\tau)}{\mathrm{d}\tau} =
  \sum_{d=1}^3 \left[
    (w^*_d -  w_d) \tilde{\phi}_{\alpha-e_d}(\tau) +
    \frac{1}{2} (\eta^* - \eta) \tilde{\phi}_{\alpha - 2e_d}(\tau)
  \right], \qquad |\alpha| \leqslant M.
\end{equation}
Considering the initial value $\tilde{\phi}_{\alpha}(0) =
\tilde{\phi}_{\alpha}$, we claim that the solution of this ODE system
is
\begin{equation}
  \label{eq:sol_odes}
  \tilde{\phi}_{\alpha}(\tau)=
    \sum_{k=0}^{|\alpha|}\tilde{\phi}_{\alpha}^{(k)}\tau^k,
  \qquad |\alpha|\leqslant M,
\end{equation}
where $\tilde{\phi}_{\alpha}^{(k)}$ is defined in
\eqref{eq:sol_odes1}. The verification of this claim is simply a direct
substitution of \eqref{eq:sol_odes1} into \eqref{eq:odes}, and the
details are omitted. Setting $\tau = 1$ in \eqref{eq:sol_odes} and
using $\tilde{\phi}_{\alpha}(1) = \tilde{\phi}^*_{\alpha}$, one
completes the proof of \eqref{eq:phi_star}.
\end{proof}

Theorem \ref{thm:projection} provides an algorithm to obtain
$\tf_{\alpha}^{[\bu,\theta]}$ from $\tf_{\alpha}$. In detail, we let
\begin{equation}
\bw = \lu, \quad \eta = \lT, \quad \bw^* = \bu, \quad \eta^* = \theta.
\end{equation}
Then by \eqref{eq:phi_star} and \eqref{eq:sol_odes1}, it can be seen
that $\tf_{\alpha}^{[\bu,\theta]}$ can be represented by a linear
combination of $\tf_{\beta}$ with $|\beta| \leqslant |\alpha|$.
Similarly, if we let
$\bw = \bu, ~\eta = \theta, ~\bw^* = \lu, ~\eta^* = \lT$, the same
technique can help us to find $\tQ_{\alpha}^*$ defined by
\begin{equation}
\tQ^*_{\alpha} = \frac{1}{\alpha!} \lT^{\frac{|\alpha|}{2}}
  \int_{\bbR^3} H_{\alpha} \left(
    \frac{\bv - \lu}{\sqrt{\lT}}
  \right) Q^*[f,f](\bv) \dd \bv,
\end{equation}
where $Q^*[f,f]$ is defined in \eqref{eq:Qstar_f}. Thus our new
collision model can be written as
\begin{equation}
Q^*[f,f](\bv) = \sum_{\alpha \in \bbN^3}
  \tQ^*_{\alpha} \mH_{\alpha}(\bv),
\end{equation}
where the map from $\tf_{\alpha}$ to $\tQ^*_{\alpha}$ can be
summarized as follows:
\begin{equation} \label{eq:map}
\tf_{\alpha} \xrightarrow{\text{Theorem \ref{thm:projection}}}
  \tf_{\alpha}^{[\bu,\theta]} \longrightarrow \eqref{eq:Qstar_f}
  \xrightarrow{\text{Theorem \ref{thm:projection}}} \tQ^*_{\alpha}.
\end{equation}

Before closing this section, we would like to mention that the choice
of the constant $\nu_{M_0}$ in \eqref{eq:Qstar} and \eqref{eq:Qstar_f}
should probably be determined by further numerical studies. Currently,
we adopt the choice in \cite{Cai2015,QuadraticCol} and set $\nu_{M_0}$
to be the spectral radius of the operator $\hat{L}_{M_0}: F_{M_0}(0,1)
\rightarrow F_{M_0}(0,1)$, whose definition is
\begin{equation}\label{eq:linearized-col}
\hat{L}_{M_0}[h](\hat{\bv}) =
  \sum_{|\alpha| \leqslant M_0} \sum_{|\beta| \leqslant M_0}
    (A_{\alpha}^{0,\beta} + A_{\alpha}^{\beta,0}) \tilde{h}_{\beta}
    \hat{\mH}_{\alpha}(\bv),
\end{equation}
which is in fact the linearization of the quadratic operator
$\hat{Q}^*$ restricted on $F_{M_0}(0,1)$. We refer the readers to
\cite{Cai2015,QuadraticCol} for more details.

\subsection{Hermite spectral method for the Boltzmann equation with
approximate collision term}
\label{sec:eqs}
Having derived the approximate collision operator $Q^*[f, f]$ in the
previous subsection, we are ready to write down the equations for the
coefficients $\tf_{\alpha}(t,\bx)$ in \eqref{eq:f_fM}. By Galerkin's
method, the equations are obtained by the following equalities:
\begin{equation}
\frac{\mfm}{\alpha!} \lT^{\frac{|\alpha|}{2}}
  \int_{\bbR^3} H_{\alpha}(\hat{\bv})
  \left[
    \frac{\partial f_M}{\partial t} + \nabla_{\bx} \cdot (\bv f_M)
  \right] \dd \bv =
\frac{\mfm}{\alpha!} \lT^{\frac{|\alpha|}{2}} \int_{\bbR^3}
  H_{\alpha}(\hat{\bv}) Q^*[f_M, f_M](\bv) \dd \bv,
\end{equation}
where $|\alpha| \leqslant M$ and $\hat{\bv}$ is defined in
\eqref{eq:changevar}. To deal with the convection term, we need the
recursion relation of the basis function $\mH_{\alpha}$:
\begin{equation}
v_j \mH_{\alpha}(\bv) = \alpha_j \mH_{\alpha-e_j}(\bv) +
  \overline{u}_j \mH_{\alpha}(\bv) + \lT \mH_{\alpha+e_j}(\bv).
\end{equation}
Thus by orthogonality of Hermite polynomials, we obtain the following
evolution equations for $\tf_{\alpha}$:
\begin{equation}
  \label{eq:moment_sys}
  \begin{aligned}
    \pdd{t}\tf_{\alpha} + \sum_{j=1}^3
    \pdd{x_j}\left((\alpha_j+1) \tf_{\indexk + e_j} +
      \overline{u}_j\tf_{\indexk} + 
      \lT\tf_{\indexk - e_j}\right)  = \tQ^*_{\alpha},
  \qquad |\alpha| \leqslant M.
  \end{aligned}
\end{equation}
where $\tf_{\beta}$ is regarded as zero if $\beta$ contains negative
indices or $|\beta| > M$. We remind the readers again that the
right-hand side of \eqref{eq:moment_sys} is a function of all
$\tf_{\alpha}$ by \eqref{eq:map}, which shows that the computation of
$\tQ^*_{\alpha}$ includes two parts:
\begin{enumerate}
\item Application of algorithm implied in Theorem \ref{thm:projection}
  (the first and third arrows in \eqref{eq:map}), whose time
  complexity is $O(M^4)$ (see \eqref{eq:phi_star} and
  \eqref{eq:sol_odes1});
\item Evaluation of all the coefficients in \eqref{eq:Qstar_f} (the
  second arrow in \eqref{eq:map}), whose time complexity is $O(M_0^9 +
  M^3)$.
\end{enumerate}
Therefore, the total time complexity for computing all
$\tQ^*_{\alpha}$ with $|\alpha| \leqslant M$ is $O(M_0^9 + M^4)$.

To complete the problem, we need to supplement \eqref{eq:moment_sys}
with initial and boundary conditions. Suppose the initial condition
for the original Boltzmann equation is $f(0,\bx,\bv) = f_0(\bx,\bv)$.
Then a natural initial condition for \eqref{eq:moment_sys} is
\begin{equation}
\tf_{\alpha}(0,\bx) = 
  \frac{\mfm}{\alpha!} \lT^{\frac{|\alpha|}{2}} \int_{\bbR^3}
    H_{\alpha} \left( \frac{\bv - \lu}{\sqrt{\lT}} \right)
    f_0(\bx,\bv) \dd \bv.
\end{equation}
The boundary condition, especially the solid wall boundary condition,
is slightly more complicated, and we will discuss this topic in
the next section.

\section{Boundary condition} \label{sec:bc}
Due to the hyperbolic nature of the Boltzmann equation, on the
boundary of the spatial domain, we need to specify the value of the
distribution function with velocity pointing into the domain. In the
simulation of microflows, the wall boundary condition is especially
important. In this paper, we focus on a popular type of boundary
condition proposed by Maxwell in \cite{Maxwell}, which is a linear
combination of the specular reflection and the diffuse reflection.
Such a boundary condition has been studied for very similar methods in
\cite{Cai2012, Microflows1D}, which make our work much easier. Below,
we are going to first review the Maxwell boundary condition, and then
propose the boundary condition for the Hermite spectral method.

\subsection{The Maxwell boundary condition for the Boltzmann equation}
Suppose $\bx_0 \in \partial \Omega$. Let $\bn_0$ be the outer unit
normal vector of the spatial domain $\Omega$ at $\bx_0$. Consider the
case in which $\bx_0$ is the contact point of the gas and the solid
wall. At point $\bx_0$, the solid wall has temperature $T^w$, and is
moving at velocity $\bu^w$. By these assumptions, the Maxwell boundary
condition is described as follows:
\begin{equation}
  \label{eq:maxwell_boundary}
  f(t, \bx_0, \bv) = 
    \omega f_{\mM}^w(t, \bx_0, \bv) + (1 - \omega)f(t,
    \bx_0, \bv^{\ast}), \qquad \text{if } (\bv - \bu^w) \cdot \bn_0 < 0,
\end{equation}
where $\omega \in [0, 1]$ is the accommodation coefficient of the wall,
and $f_{\mM}^w$ and $\bv^{\ast}$ are defined as 
\begin{equation}
  \label{eq:fwall}
  f_{\mM}^w(t, \bx_0, \bv) = \rho^w\mM_{\bu^w, \theta^w}(\bv), \qquad
  \bv^{\ast} = \bv - 2[(\bv - \bu^w)\cdot \bn_0]\bn_0. 
\end{equation}
In \eqref{eq:fwall}, $\rho^w$ should be determined by the condition
that the normal mass flux on the boundary is zero, that is 
\begin{equation}
  \label{eq:mass_conser}
  \int_{(\bv - \bu^w) \cdot \bn_0 < 0} \left[(\bv - \bu^w) \cdot
    \bn_0\right]
  \left[f_{\mM}^w(t, \bx_0, \bv) - f(t, \bx_0, \bv^{\ast})\right] \dd \bv = 0.
\end{equation}
The boundary condition for Hermite spectral method should be an
approximation of the above boundary condition.

\subsection{Boundary condition for the Hermite spectral method}
The most natural idea to find the boundary conditions for
\eqref{eq:moment_sys} is to integrate the Maxwell boundary condition
\eqref{eq:maxwell_boundary} against Hermite polynomials. However, if
all the Hermite polynomials of degree less than or equal to $M$ are
taken into account, the resulting number of boundary conditions will
generally be larger than the number required by the hyperbolicity. In
general, for a hyperbolic system, the number of boundary conditions at
$\bx_0 \in \Omega$ should be equal to the number of characteristics
pointing into the domain $\Omega$. Below we will first find all the
characteristic speeds of the system.

For clarification purposes, we rewrite \eqref{eq:moment_sys} in the
matrix-vector form:
\begin{equation}
    \label{eq:moment-sys-matrix}
\pd{\bdf}{t} + \sum_{j=1}^3 {\bf A}_j \pd{\bdf}{x_j} =
  \boldsymbol{Q}(\bdf),
\end{equation}
where $\bdf$ is a column vector with all the unknowns $\tf_{\alpha}$,
$|\alpha| \leqslant M$ as its components, and ${\bf A}_j$ and
$\boldsymbol{Q}$ are respectively defined by the convection and
collision terms in \eqref{eq:moment_sys}.

We first consider the case $\bn_0 = (1, 0, 0)^T$, in which it is only
necessary to find all the eigenvalues of ${\bf A}_1$. The matrix ${\bf
A}_1$ is in fact a reducible matrix, which can be observed if we
divide $\bdf$ into the following subvectors:
\begin{equation}
\bdf_{\alpha'} = \left(
  \tf_{0,\alpha'}, \tf_{1,\alpha'}, \cdots, \tf_{M-|\alpha'|,\alpha'}
\right)^T, \qquad
  \alpha' = (\alpha_1', \alpha_2') \in \bbN^2, \quad
  |\alpha'| \leqslant M.
\end{equation}
Here the notation $\tf_{k,\alpha'}$ designates the coefficient
$\tf_{\alpha}$ with $\alpha = (k, \alpha_1', \alpha_2')$. Apparently,
the vector $\bdf$ can be formed by gluing up $\bdf_{\alpha'}$ for all 
$\alpha' \in \bbN^2$. Thus by \eqref{eq:moment_sys}, one can find that
${\bf A}_1$ has a block-diagonal structure, and each block has a
tridiagonal form
\begin{equation}
{\bf A}_{1\alpha'} = \begin{pmatrix}
  \overline{u}_1 & 1 \\
  \overline{\theta} & \overline{u}_1 & 2 \\
  & \overline{\theta} & \overline{u}_1 & 3 \\
  & & \ddots & \ddots & \ddots \\
  & & & \overline{\theta} & \overline{u}_1 & M-|\alpha'| \\
  & & & & \overline{\theta} & \overline{u}_1
\end{pmatrix}, \qquad
  \alpha' \in \bbN^2, \qquad |\alpha'| \leqslant M.
\end{equation}
All the eigenvalues of the above matrix has been given in
\cite{Cai2015Framework} as\footnote{In \cite{Cai2015Framework}, such a
matrix is denoted as ${\bf M}(\overline{u}_1, \overline{\theta})$. It
is shown in \cite{Cai2015Framework} that this matrix is similar to a
diagonal matrix called ${\bf \Lambda}(\bw)$, whose diagonal entries
are exactly the numbers given in \eqref{eq:ev}.}
\begin{equation} \label{eq:ev}
\lambda({\bf A}_{1\alpha'}) = \left\{
  \overline{u}_1 + c_0 \sqrt{\lT},
  \quad \overline{u}_1 + c_1 \sqrt{\lT},
  \quad \cdots, \quad \overline{u}_1 + c_{M-|\alpha'|} \sqrt{\lT}
\right\},
\end{equation}
where $c_0, \cdots, c_{M-|\alpha'|}$ are all the roots of the
one-dimensional Hermite polynomial of degree $M+1-|\alpha'|$.
Consequently, all the eigenvalues of ${\bf A}_1$ are given by
\begin{equation}
\lambda({\bf A}_1) =
  \bigcup_{\substack{\alpha' \in \bbN^2\\ |\alpha'| \leqslant M}}
  \lambda({\bf A}_{1\alpha'}).
\end{equation}

When $\bn_0 = (1,0,0)^T$, the number of boundary conditions at $\bx_0$
should equal the number of eigenvalues less than $u_1^w$. In general,
this number varies with $\overline{u}_1$ and $\lT$, which makes it
difficult to discuss the boundary conditions in the general
setting. As a workaround, we assume that $\lu$ is chosen such that
$\overline{u}_1 = u_1^w$. Then, by the symmetry of the Hermite
polynomials, the number of boundary conditions to be specified at
$\bx_0$ is
\begin{equation} \label{eq:number_bc}
\sum_{\substack{\alpha' \in \bbN^2\\ |\alpha'| \leqslant M}}
  \left\lceil \frac{M-|\alpha'|}{2} \right\rceil.
\end{equation}
In \cite{Microflows1D}, it is proven that the number
\eqref{eq:number_bc} equals the number of indices in the following
index set:
\begin{equation}
\mathcal{A} = \{ \alpha \in \bbN^3 \mid
  |\alpha| \leqslant M, \ \alpha_1 \text{ is odd} \}.
\end{equation}
Therefore, as stated in the beginning of this section, all the
boundary conditions can be formulated by
\begin{equation} \label{eq:bc}
\begin{split}
& \frac{m}{\alpha!} \lT^{\frac{|\alpha|}{2}}
  \int_{(\bv-\bu^w)\cdot \bn_0 < 0}
  H_{\alpha} \left( \frac{\bv - \lu}{\sqrt{\lT}} \right)
  f_M(t,\bx_0,\bv) \dd\bv = \\
& \quad \frac{m}{\alpha!} \lT^{\frac{|\alpha|}{2}}
  \int_{(\bv-\bu^w)\cdot \bn_0 < 0}
  H_{\alpha} \left( \frac{\bv - \lu}{\sqrt{\lT}} \right)
  [\omega f_{\mM}^w(t,\bx_0,\bv) + (1-\omega) f_M(t,\bx_0,\bv^*)] \dd\bv
\end{split}
\end{equation}
for all $\alpha \in \mathcal{A}$. Here the function $f$ in
\eqref{eq:mass_conser} should be changed to $f_M$ when defining the
``wall Maxwellian'' $f_{\mM}^w$. The boundary conditions given by
\eqref{eq:bc} also agree with Grad's idea of using ``odd moments''
to ensure the continuity of the boundary conditions with respect to
the accommodation coefficient $\omega$. We refer the readers to
\cite{Grad} for more details.

For a general normal vector $\bn_0$, the above method still applies.
We need to assume $\lu \cdot \bn_0 = \bu^w \cdot \bn_0$, and replace
$H_{\alpha}\left( (\bv - \lu)/\sqrt{\lT} \right)$ by
$H_{\alpha}\left( {\bf R} (\bv - \lu)/\sqrt{\lT} \right)$ in
\eqref{eq:bc}, where $\bf R$ is a rotation matrix satisfying ${\bf R}
\bn_0 = (1,0,0)^T$. The remaining task is just to evaluate the
intergrals in \eqref{eq:bc}. For the case $\bn_0 = (1,0,0)^T$, this
has been done in \cite{Microflows1D}. The results are
\begin{equation}
  \label{eq:boundary_coe}
  \begin{aligned}
    \tf_{\alpha} = \frac{2\omega}{2 - \omega} & \left[
      \sqrt{\frac{2\pi}{\theta^w}}
      \hat{J}_{\alpha_1}J_{\alpha_2}(u_2^w-u_2)J_{\alpha_3}(u_3^w-u_3)
      \sum_{k=0}^{\lfloor M/2 \rfloor} S(1, 2k)\lT^{1/2-k}\tf_{2ke_1}
    \right.  \\
    & \quad \left. + \sum_{k=0}^{K(\alpha)}S(\alpha_1,
      2k)\lT^{\alpha_1/2-k} \tf_{\alpha + (2k - \alpha_1)e_1}\right],
    \quad \alpha \in \mathcal{A}.
  \end{aligned}
\end{equation}
where $K(\alpha) = \lfloor(M-\alpha_2 - \alpha_3) /2\rfloor$, and
$J_r(\cdot)$ and $\hat{J}_r(\cdot)$ are recursively defined by
\begin{gather*}
  J_{-1}(u)= 0, \qquad J_0(u) = 1, \qquad 
  J_r(u) = \frac{1}{r}[(\theta^w - \lT)J_{r-2}(u) + uJ_{r-1}(u)],
    \quad r\geqslant 1; \\
  S_0 = 0,    \qquad S_1 =
  \sqrt{\frac{\theta^w}{2\pi}}, \qquad  S_r = -\frac{r-2}{r(r-1)}\lT
  S_{r-2}, \quad r \geqslant 2; \\
  \hat{J}_{-1} = 0, \qquad \hat{J}_0 = 1/2,  \qquad  
  \hat{J}_r = \frac{1}{r}(\theta^w - \lT)\hat{J}_{r-2} - S_r,
  \quad r \geqslant 1. 
\end{gather*}
To define $S(\cdot,\cdot)$, we first introduce $K(\cdot,\cdot)$ by
\begin{equation}
  \label{eq:K}
  K(r, s) =\left\{
  \begin{array}{ll}
    \frac{(-1)^{(r + s -1)/2}\sqrt{2\pi}(s-1)!!}{r
    2^{(r-1)/2}\left(\frac{r-1}{2}\right)!},
    & r \text{ is odd and } s \text{ is even},  \\
    0, & \text{otherwise},
  \end{array}\right.
\end{equation}
which makes it convenient to define $S(\cdot, \cdot)$:
\begin{equation}
  \label{eq:p_s}
  S(r, s) = \left\{ 
    \begin{array}{ll}
      1/2, & r = s= 0, \\
      K(1, s-1), & r = 0 \text{ and } s\neq 0, \\
      K(r, 0), & r \neq 0 \text{ and } s = 0, \\
      K(r, s) + S(r -1, s-1) s / r, & \text{otherwise}.
    \end{array}
\right.
\end{equation}
It can be verified that when $\alpha = e_1 = (1,0,0)$, the boundary
condition \eqref{eq:boundary_coe} can be simplified as $\tf_{e_1} =
0$, which indicates that the mass flux on the boundary is zero. In
this paper, such a special case ($\bn_0 = (1,0,0)^T$) is sufficient
for our numerical experiments. General discussions on the
implementation of boundary conditions will be left for the future
work.


\section{Numerical algorithms and experiments}
\label{sec:num}
Numerical algorithms to solve the system \eqref{eq:moment-sys-matrix}
with the boundary condition \eqref{eq:boundary_coe} are briefly
introduced in this section. Two spatially one-dimensional problems
and a spatially two-dimensional problem, with the convenient setting
\begin{equation}
\pd{\bdf}{x_2} = \pd{\bdf}{x_3} \equiv 0 \quad \text{and} \quad
\pd{\bdf}{x_3} \equiv 0
\end{equation}
respectively,
are then given to illustrate the effectiveness of the proposed solver.
Here $\bdf$ is still the vector of coefficients for a
three-dimensional distribution function.

\subsection{Numerical algorithm}
\label{sec:num-alg}
Suppose the spatial domain $\Omega \subset \mathbb{R}^{N}$ is
discretized by a uniform grid with cell size $\Delta x$ and cell
centers $\bx_j=(x_{j_1},\ldots,x_{j_N})$, $j\in \mathbb{Z}^N$. Using
$\bdf_{j}^{n}$ to approximate the average of $\bdf$ over the $j$th
grid cell $[x_{j_{1}-1/2}, x_{j_{1}+1/2}] \times \cdots \times
[x_{j_{N}-1/2}, x_{j_{N}+1/2}]$ at time $t^{n}$, the system
\eqref{eq:moment-sys-matrix} can be solved by Euler's method with time
step size $\Delta t$ as following:
\begin{equation}
  \label{eq:numerical-scheme}
  \bdf_{j}^{n+1} = \bdf_{j}^{n} - \frac{\Delta t}{\Delta x}
\sum_{d=1}^{N} \left[ \boldsymbol{F}_{j+\frac{1}{2} e_{d}}^{n} -
\boldsymbol{F}_{j-\frac{1}{2}e_{d}}^{n} \right] + \Delta t
\boldsymbol{Q}(\bdf_{j}^{n}),
\end{equation}
where the finite volume method is employed for spatial discretization,
and $\boldsymbol{F}_{j+\frac{1}{2}e_{d}}^{n}$ is the numerical flux at
the boundary between the cells with center $\bx_j$ and $\bx_{j+e_d}$.
In the present experiments, the HLL flux \cite{HLL}, given by
\begin{align}
  \label{eq:HLL-flux}
  \boldsymbol{F}_{j+\frac{1}{2}e_{d}}^{n} = \left\{
    \begin{aligned}
      & {\bf A}_{d} \bdf_{j+\frac{1}{2}e_{d}}^{n, L}, & & \lambda_{d}^{L} \geq 0, \\
      & \frac{\lambda_{d}^{R} {\bf A}_{d} \bdf_{j+\frac{1}{2}e_{d}}^{n,L} -
        \lambda_{d}^{L} {\bf A}_{d} \bdf_{j+\frac{1}{2}e_{d}}^{n,R} + \lambda_{d}^{R}
        \lambda_{d}^{L} (\bdf_{j+\frac{1}{2}e_{d}}^{n,R} -
        \bdf_{j+\frac{1}{2}e_{d}}^{n,L})}{\lambda_{d}^{R} - \lambda_{d}^{L}}, \!\!\! &&
      \lambda_{d}^{L} < 0 < \lambda_{d}^{R}, \\
      & {\bf A}_{d} \bdf_{j+\frac{1}{2}e_{d}}^{n,R}, & & \lambda_{d}^{R} \leq 0,
    \end{aligned}
  \right.
\end{align}
is adopted. Here $\lambda_{d}^{L} = \overline{u}_{d} - C_{M+1}\sqrt{\lT}$
and $\lambda_{d}^{R} = \overline{u}_{d} + C_{M+1} \sqrt{\lT}$, where
$C_{M+1}$ is the maximal root of the Hermite polynomial of degree
$M+1$. In our experiments, $\lu$ will be set to be $0$, and thus only
the middle case of \eqref{eq:HLL-flux} is active. The scheme
\eqref{eq:numerical-scheme} can be improved straightforwardly to
higher-order temporal schemes by Runge-Kutta methods. In order to get
second-order spatial accuracy, the approximate solutions on the cell
boundary $\bdf_{j-\frac{1}{2}e_{d}}^{n,R}$ and $\bdf_{j+\frac{1}{2}e_{d}}^{n,L}$ are computed by
the linear reconstruction
\begin{equation}
  \label{eq:linear-recon}
  \bdf_{j-\frac{1}{2}e_{d}}^{n,R} = \bdf_{j}^{n} -\frac{1}{2} \Delta x \bg_{d}^{n},
\quad \bdf_{j+\frac{1}{2}e_{d}}^{n,L} = \bdf_{j}^{n} + \frac{1}{2} \Delta x \bg_{d}^{n}, 
\end{equation}
with
$\bg_{d}^{n} = \frac{1}{2} \left(\bdf_{j+e_{d}}^{n} - \bdf_{j-e_{d}}^{n}
\right)/\Delta x$.

In our numerical experiments, we are interested in the steady state of
microflows. However, due to the stability restriction of the explicit
time-stepping scheme, the time step size should be chosen to satisfy
the CFL condition
\begin{equation}
  \label{eq:CFL}
  \Delta t \sum_{d=1}^{N}\frac{\vert\overline{u}_{d}\vert + C_{M+1} \sqrt{\lT}}{\Delta x} < 1,
\end{equation}
which indicates a long time simulation would be taken to achieve the
steady state. Such a method will be used in our two-dimensional
examples to be shown in Section \ref{sec:2D}. For one-dimensional
steady-state problems, several additional techniques can be taken
into account to accelerate the simulation by giving up the time
accuracy of the solution. A simple way is to revise the computation of
\eqref{eq:numerical-scheme} on the whole spatial domain from the
Jacobi-type iteration into a cell-by-cell symmetric Gauss-Seidel (SGS)
iteration as shown in \cite{hu2016acceleration}. The SGS iteration is
in general several times faster than the explicit time-stepping
scheme, although for both methods, the total number of iterations is
expected to grow linearly as the grid number increases.

Further acceleration of the steady-state computation can be obtained
by using the multigrid technique, which has been explored in
\cite{hu2014nmg, hu2015}. The same framework of the nonlinear
multigrid method as proposed in \cite{hu2014nmg} is used in our
simulation, except that the single level iteration is replaced by the
above SGS iteration. By noting that $\lu$ and $\lT$ are
constants,\footnote{In \cite{hu2014nmg}, basis functions vary
spatially. The idea has been sketched in Section \ref{sec:intro} and 
the implementation is more difficult due to the nonlinearity.} the
implementation is in fact much easier than that in \cite{hu2014nmg}.

\subsection{One-dimensional numerical experiments}
\label{sec:num-exp}
Numerical experiments of the planar Couette flow and the Fourier flow
are carried out below. Numerical solutions of the quadratic collision
term \eqref{eq:Qstar_f} as well as its linearization
\eqref{eq:linearized-col} are provided. In all simulations, a uniform
grid with $256$ cells is used for spatial discretization, and the gas
of argon, which has molecule mass $\mfm = 6.63\times 10^{-26}{\rm kg}$
and molecule diameter $d_{\rm ref} = 4.17 \times 10^{-10}{\rm m}$ at
the reference temperature $T_{\rm ref} = 273.15{\rm K}$, is
considered. The Maxwellian with density $\rho = 9.282\times
10^{-6}{\rm kg} \cdot {\rm m}^{-3}$, velocity $\bu = 0 {\rm m/s}$ and
temperature $T=273.15{\rm K}$ is adopted to set the initial value of
the simulation. In order to match the reference results produced by
the DSMC method \cite{Bird}, the viscosity coefficient $\mu$ used in
the collision term \eqref{eq:Qstar_f} is set to be
\begin{equation}
  \mu = \frac{60 (\eta-1)^{2} \sqrt{\mfm k_{B} T_{\rm
ref}/\pi}}{(\eta-2)(3\eta-5)d_{\rm ref}^{2}} \cdot
\left(\frac{T}{T_{\rm ref}}\right)^{\frac{1}{2}(\eta+3)/(\eta-1)},
\end{equation}
where the Boltzmann constant
$k_{B} = 1.380658 \times 10^{-23} {\rm m^{2}\cdot kg \cdot s^{-2}
  \cdot K^{-1}}$, and the index $\eta$ is set to be $10$.

\subsubsection{The planar Couette flow}
\label{sec:num-exp-couette}

Consider the gas between two infinite parallel plates, which have the
temperature $T^{w}=273.15{\rm K}$, and move in the opposite direction
along the plate with the speed $119.25{\rm m/s}$. Both plates are
assumed to be completely diffusive, which indicates the accommodation
coefficient $\omega = 1$ in the boundary condition. Driven by the
motion of the plates, the flow will reach a steady state as time tends
to infinity. Numerically, we let the computational domain be
$[-D/2, D/2]$, where $D$ is the distance between the two plates. Four
choices of the distance, i.e., $D=0.092456 {\rm m}$,
$0.018491 {\rm m}$, $D = 0.003698 {\rm m}$ and $0.00074 {\rm m}$,
corresponding to the dimensionless Knudsen number $\Kn = 0.1$, $0.5$,
$2.5$ and $12.5$ respectively, are investigated.  Additionally, we set
$\lu = 0$ and $\lT = \frac{k_{B}}{m}T_{\rm ref}$ in this example.

(1) $D=0.092456{\rm m}$, $\Kn=0.1$: Numerical results for the
quadratic collision term \eqref{eq:Qstar_f} with $M_{0}=5$, as well as
the DSMC solutions, are listed in
\figurename~\ref{fig:couette-Kn01-binaryIM5}. Only half of the domain
is plotted, by noting that the density, the temperature and the shear
stress are even functions, and the heat flux is an odd function. Fast
convergence of these quantities is observed as $M$ increases. All
results coincide very well with the DSMC results. Note that the
actually relative error of shear stress $\sigma_{12}$ is less than
$1.5\%$ even for the worst case $M=5$, although an evident deviation
can be seen from the figure. It turns out that a small $M$, e.g.,
$M=5$, with $M_{0}=5$ for the quadratic collision term
\eqref{eq:Qstar_f} is enough to give satisfactory results in this
case. In fact, even for the linearized collision term
\eqref{eq:linearized-col} with $M_{0}=5$, numerical results also agree
well with the results shown in
\figurename~\ref{fig:couette-Kn01-binaryIM5}, except that a slight
deviation can be observed for temperature. The comparison of
temperature profiles between the quadratic collision term
\eqref{eq:Qstar_f} and its linearization \eqref{eq:linearized-col} can
be found in
\figurename~\ref{fig:couette-Kn01-linearizedIM5-temperature}, from
which one can see that the quadratic form provides more accurate
description of the fluid states.

(2) $D=0.018491{\rm m}$, $\Kn=0.5$: As the Knudsen number gets larger,
larger $M$ is necessary to be considered. Numerical results for the
quadratic collision term \eqref{eq:Qstar_f} with $M_{0}=5$, as well as
the DSMC solutions, are shown in
\figurename~\ref{fig:couette-Kn05-binaryIM5}. Again, only half of the
domain is displayed. In this case, significant deviation can be
observed between solutions with small $M$ and the DSMC solutions. And
the solutions behave differently for odd and even $M$, as exhibited
many times in the literature (see e.g. \cite{Cai2018, Microflows1D}).
In spite of this, the convergence can still be obtained for all plotted
quantities, and they match the DSMC solutions better as $M$ increases.
Nevertheless, the quadratic collision term \eqref{eq:Qstar_f} with
$M_{0}=5$ still seems to be sufficient for Knudsen number $0.5$, as
long as $M$ is sufficiently large.

Numerical results for the linearized collision term
\eqref{eq:linearized-col} with $M_{0}=10$, which is expected better
than the same collision term with $M_{0}=5$, are presented in
\figurename~\ref{fig:couette-Kn05-linearizedIM5} for comparison.
Although convergence of these quantities is also observed with
respect to $M$, the results are not as good as those obtained by
quadratic collision term with $M_{0}=5$ and the same $M$. More
precisely, there is a significant gap between the possible limiting
temperature and the reference temperature given by the DSMC
method. This indicates that the linearized collision term is indeed
inadequate for problems with such a Knudsen number.

(3) $D = 0.003698 {\rm m}$, $\Kn = 2.5$: This example tests our
collision model for the flow in the transitional regime. Since the
Knudsen number is even larger, we consider only the quadratic
collision term \eqref{eq:Qstar_f} with $M_0 = 10$. The comparison
between our results and the DSMC solutions are provided in
\figurename~\ref{fig:couette-Kn25-binaryIM10}. It shows that the
Hermite spectral method still provides high-quality solutions for
lower-order moments such as density, temperature and shear stress.
Precisely speaking, for density, the relative deviation between the
DSMC solution and our solution is lower than $0.03\%$, and the
relative deviation of temperature and shear stress is less than
$0.05\%$ and $0.9\%$ respectively. For the heat flux, our solution
agrees well with DSMC results when the flow is away from the boundary,
while obvious discrepancy can be observed near the boundary, where and
the relative deviation is close to $9\%$.

(4) $D = 0.00074 {\rm m}$, $\Kn = 12.5$: For such a high Knudsen
number, the flow is in the free molecular regime. The strong
nonequilibrium requires an accurate modelling of the collision term to 
precisely capture the flow structure. Again we only present the
results for the quadratic collision term \eqref{eq:Qstar_f} with $M_0
= 10$. Our numerical results and DSMC solutions are shown in
\figurename~\ref{fig:couette-Kn125-binaryIM10}, where we see that our
solutions are comparable with the DSMC solution for the density,
temperature and shear stress with $M = 65$. For the density, the
relative deviation between our solution and the DSMC solution is
$0.02\%$, and for temperature and shear stress, the relative errors
are less than $0.3\%$ and $1\%$, respectively. However, the structure
of heat flux is not well captured. The relative deviation is around
$30\%$. This may be because when the Knudsen number is large, a sharp
discontinuity exists in the distribution function, which causes Gibbs
phenomenon when the distribution function is approximated using the
spectral method. Thus, the spectral Galerkin method becomes
inefficient. A similar observation is also presented in
\cite{LWuComparative2017}.

\begin{figure}[!htb]
  \centering
  \subfloat[Density, $\rho~({\rm kg\cdot m^{-3}})$]{\includegraphics[width=0.49\textwidth,clip]{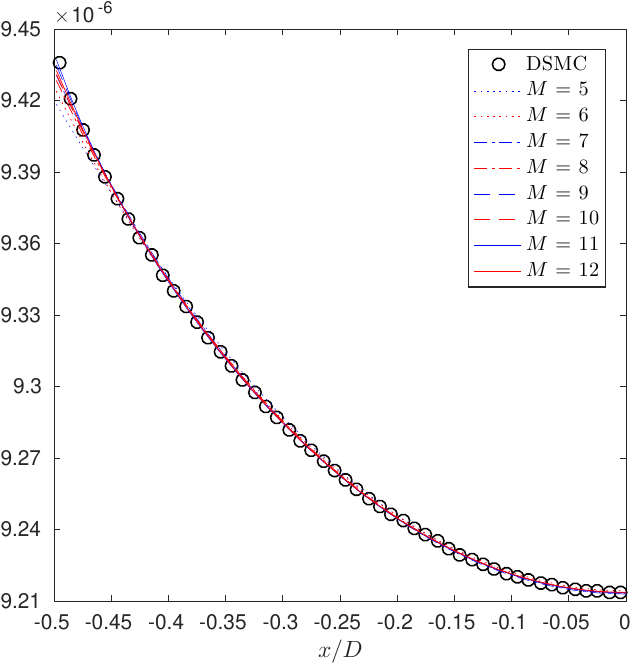}}\hfill
  \subfloat[Temperature, $T~({\rm K})$]{\includegraphics[width=0.49\textwidth,clip]{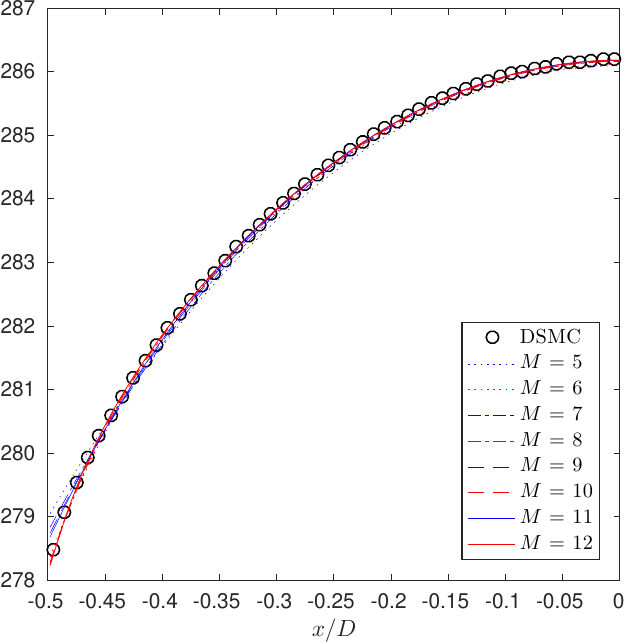}} \\
  \subfloat[Shear stress, $\sigma_{12}~({\rm kg \cdot m^{-1}\cdot s^{-2}})$]{\includegraphics[width=0.51\textwidth,clip]{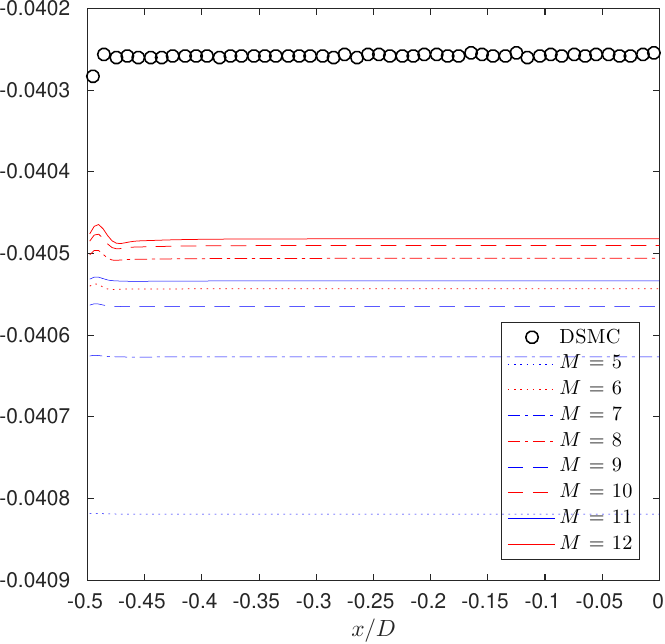}} \hfill
  \subfloat[Heat flux, $q_1~(\rm kg/s^{3})$]{\includegraphics[width=0.48\textwidth,clip]{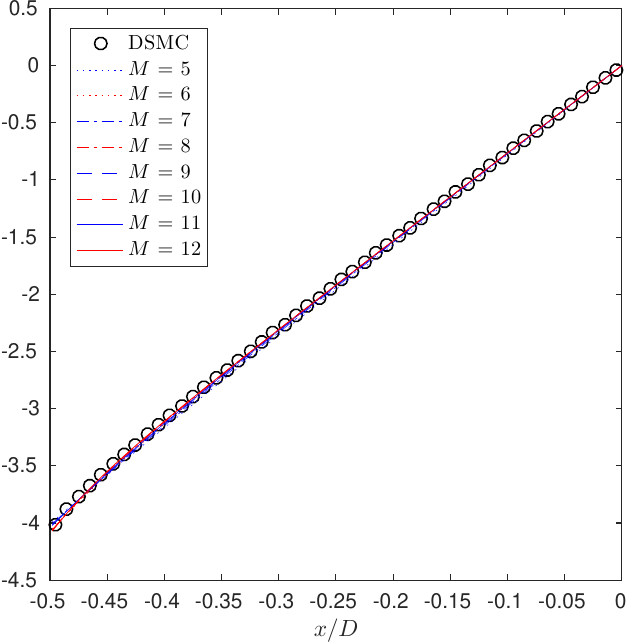}} 
\caption{Solution of the Couette flow for the quadratic collision term
  \eqref{eq:Qstar_f} with $M_{0}=5$ and $D=0.092456 {\rm
    m}$ ($\Kn=0.1$).} 
  \label{fig:couette-Kn01-binaryIM5}
\end{figure}

\begin{figure}[!htb]
  \centering
  \includegraphics[width=0.475\textwidth,clip]{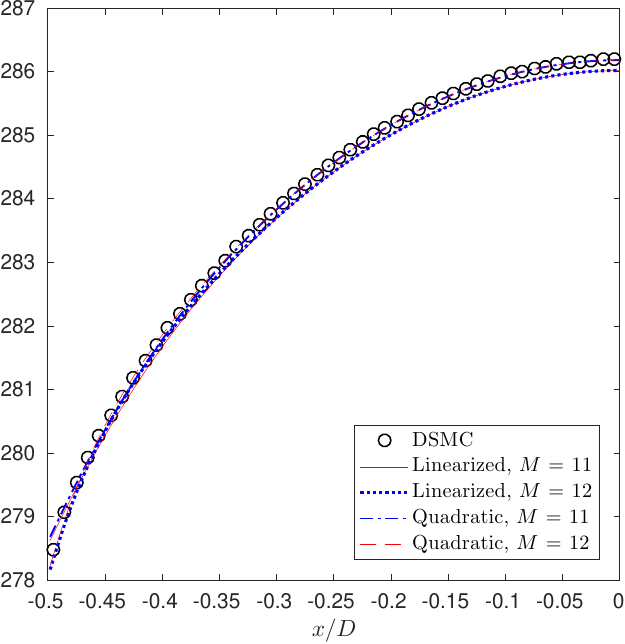}\hfill
  \includegraphics[width=0.49\textwidth,clip]{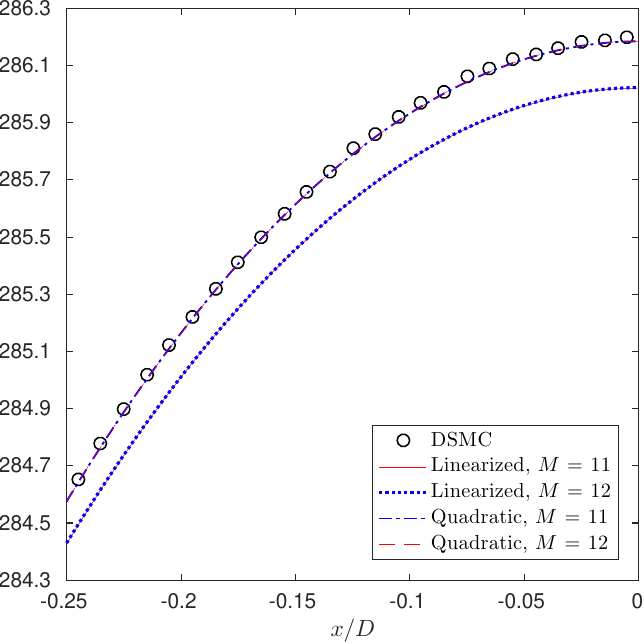}
  \caption{Comparison of temperature $({\rm K})$ profiles (left) and
    its zoom (right) between the quadratic collision term
    \eqref{eq:Qstar_f} and its linearization \eqref{eq:linearized-col}
    with $M_{0}=5$ and $D=0.092456 {\rm m}$
    ($\Kn=0.1$).} 
  \label{fig:couette-Kn01-linearizedIM5-temperature}
\end{figure}

\begin{figure}[!htb]
  \centering
  \subfloat[Density, $\rho~({\rm kg\cdot m^{-3}})$]{\includegraphics[width=0.49\textwidth,clip]{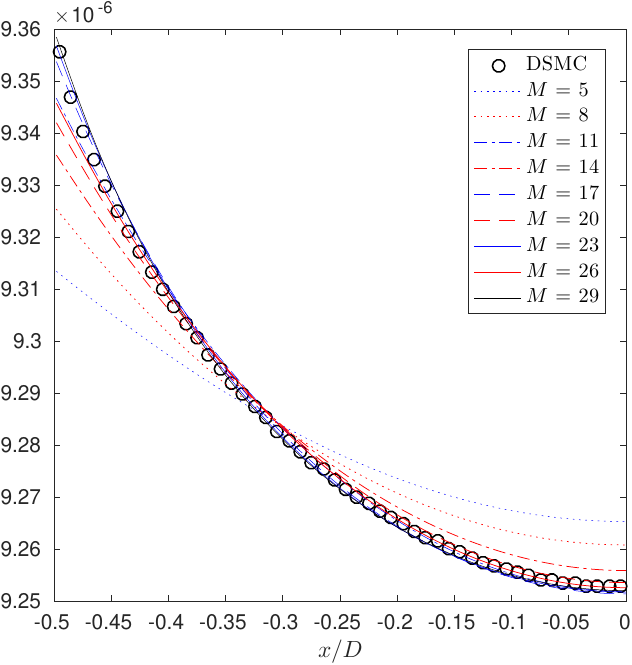}}\hfill
  \subfloat[Temperature, $T~({\rm K})$]{\includegraphics[width=0.49\textwidth,clip]{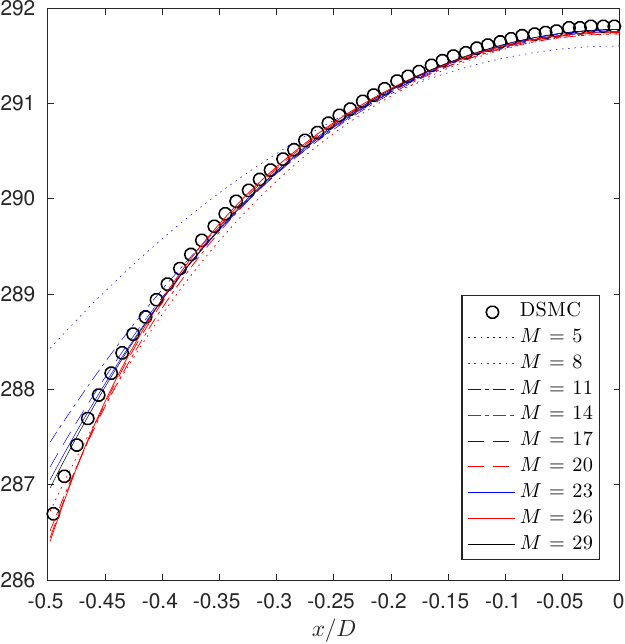}} \\
  \subfloat[Shear stress, $\sigma_{12}~({\rm kg \cdot m^{-1}\cdot s^{-2}})$]{\includegraphics[width=0.51\textwidth,clip]{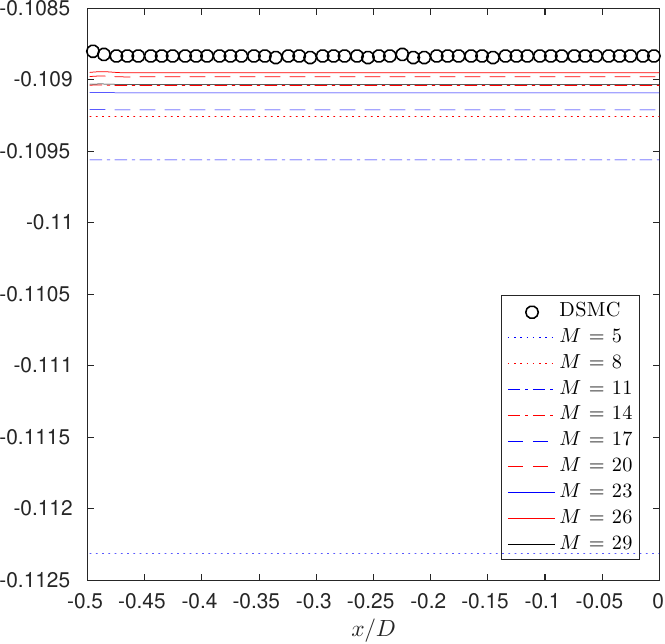}} \hfill
  \subfloat[Heat flux, $q_1~(\rm kg/s^{3})$]{\includegraphics[width=0.47\textwidth,clip]{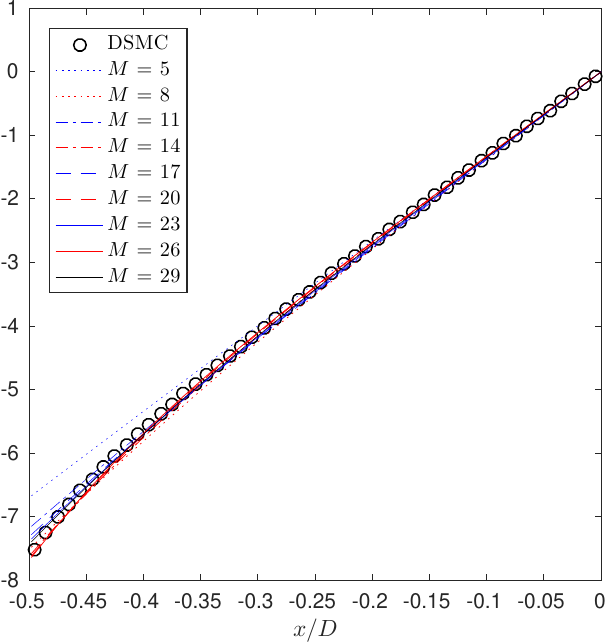}} 
\caption{Solution of the Couette flow for the quadratic collision term
  \eqref{eq:Qstar_f} with $M_{0}=5$ and $D=0.018491 {\rm
    m}$ ($\Kn=0.5$).} 
  \label{fig:couette-Kn05-binaryIM5}
\end{figure}

\begin{figure}[!htb]
  \centering
  \subfloat[Density, $\rho~({\rm kg\cdot m^{-3}})$]{\includegraphics[width=0.49\textwidth,clip]{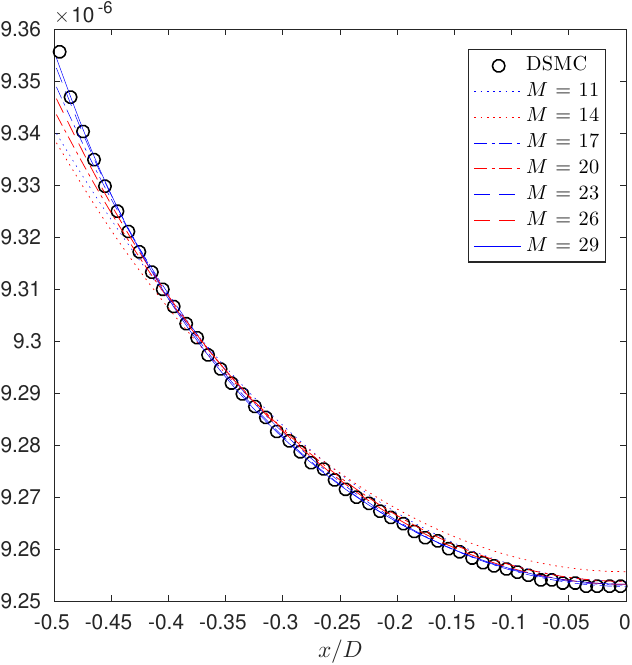}}\hfill
  \subfloat[Temperature, $T~({\rm K})$]{\includegraphics[width=0.49\textwidth,clip]{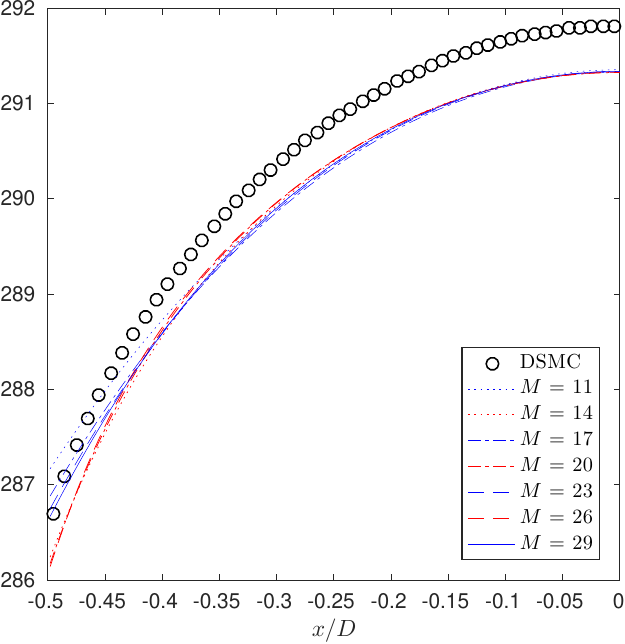}} \\
  \subfloat[Shear stress, $\sigma_{12}~({\rm kg \cdot m^{-1}\cdot s^{-2}})$]{\includegraphics[width=0.51\textwidth,clip]{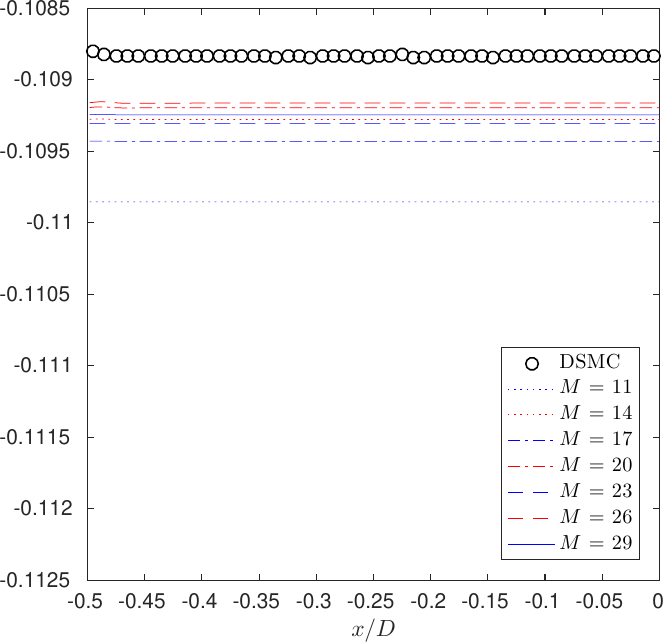}} \hfill
  \subfloat[Heat flux, $q_1~(\rm kg/s^{3})$]{\includegraphics[width=0.47\textwidth,clip]{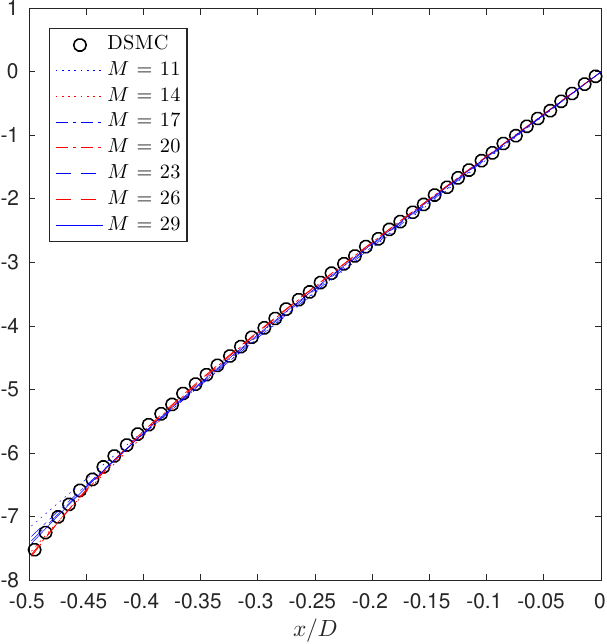}} 
\caption{Solution of the Couette flow for the linearized collision term
  \eqref{eq:linearized-col} with $M_{0}=10$ and $D=0.018491 {\rm
    m}$ ($\Kn=0.5$).} 
  \label{fig:couette-Kn05-linearizedIM5}
\end{figure}

\begin{figure}[!htb]
  \centering
  \subfloat[Density, $\rho~({\rm kg\cdot m^{-3}})$]{\includegraphics[width=0.49\textwidth,clip]{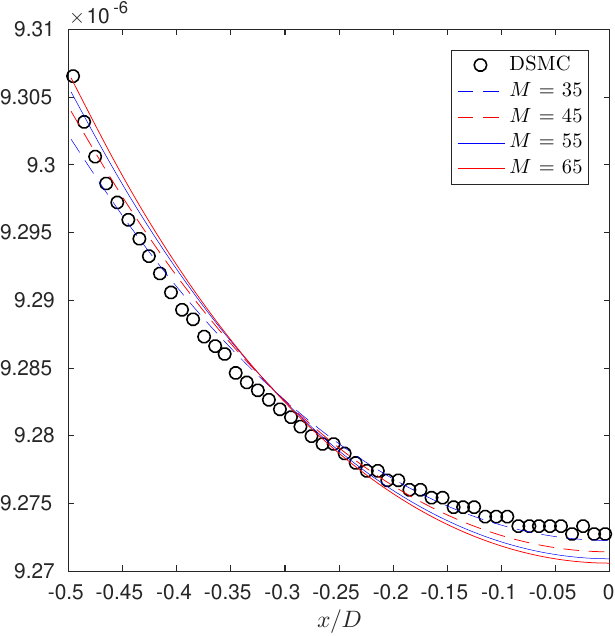}}\hfill
  \subfloat[Temperature, $T~({\rm K})$]{\includegraphics[width=0.49\textwidth,clip]{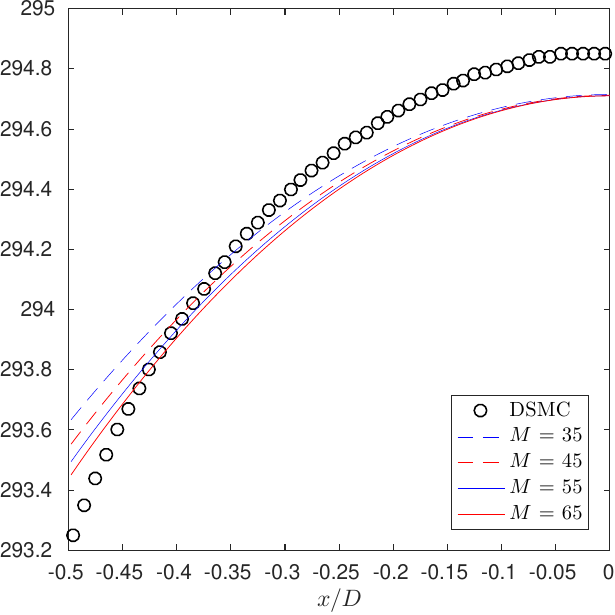}} \\
  \subfloat[Shear stress, $\sigma_{12}~({\rm kg \cdot m^{-1}\cdot s^{-2}})$]{\includegraphics[width=0.51\textwidth,clip]{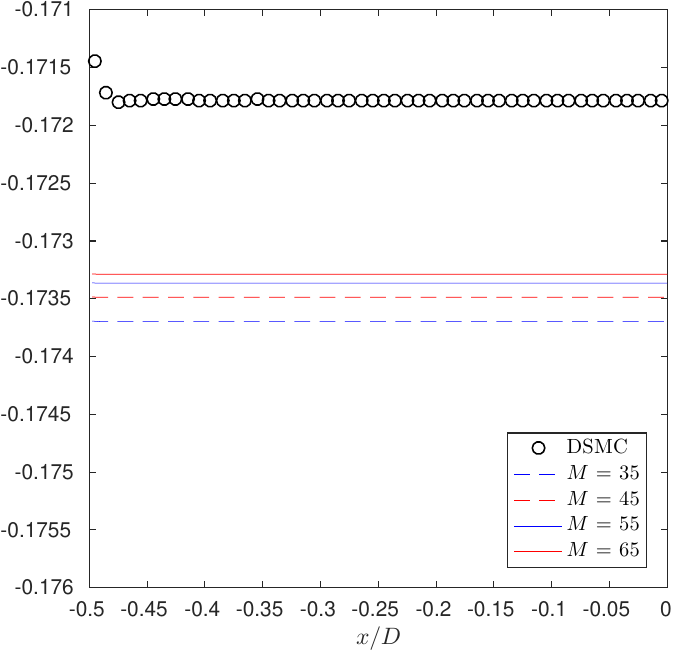}} \hfill
  \subfloat[Heat flux, $q_1~(\rm kg/s^{3})$]{\includegraphics[width=0.47\textwidth,clip]{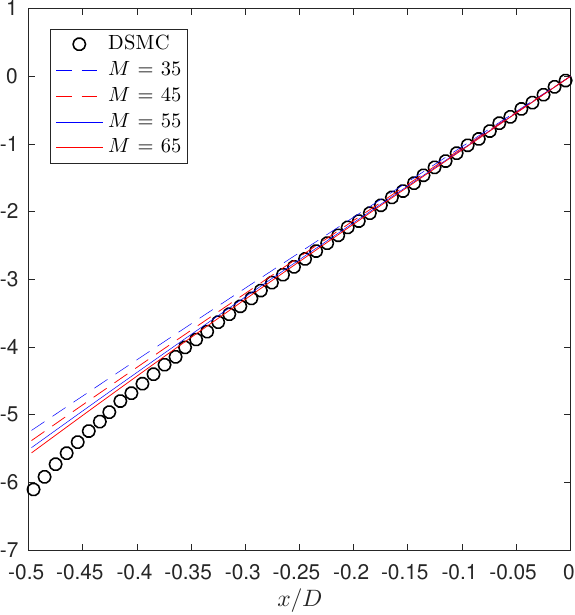}} 
\caption{Solution of the Couette flow for the quadratic collision term
  \eqref{eq:Qstar_f} with $M_{0}=10$ and $D=0.003698 {\rm
    m}$ ($\Kn=2.5$).} 
  \label{fig:couette-Kn25-binaryIM10}
\end{figure}

\begin{figure}[!htb]
  \centering
  \subfloat[Density, $\rho~({\rm kg\cdot m^{-3}})$]{\includegraphics[width=0.49\textwidth,clip]{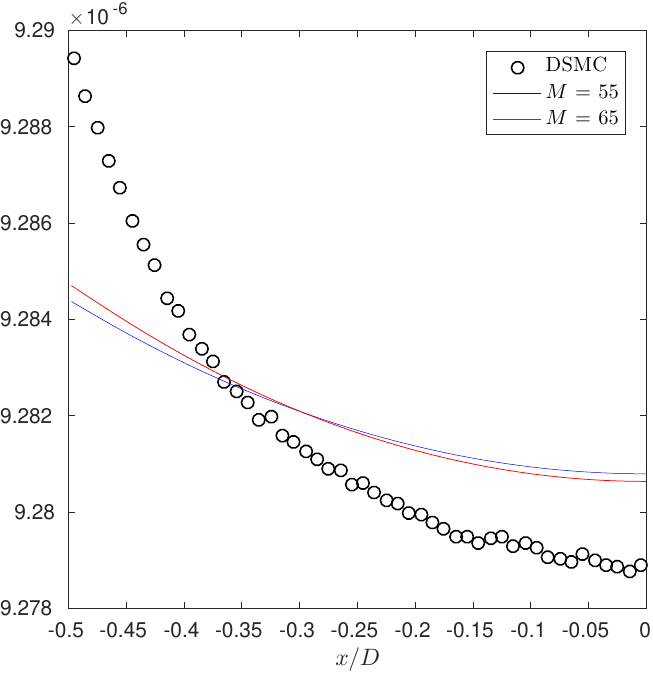}}\hfill
  \subfloat[Temperature, $T~({\rm K})$]{\includegraphics[width=0.49\textwidth,clip]{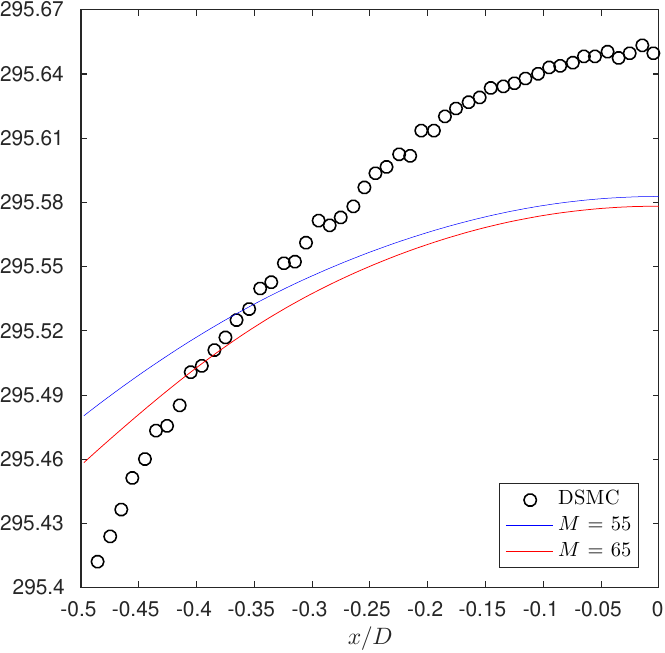}} \\
  \subfloat[Shear stress, $\sigma_{12}~({\rm kg \cdot m^{-1}\cdot s^{-2}})$]{\includegraphics[width=0.51\textwidth,clip]{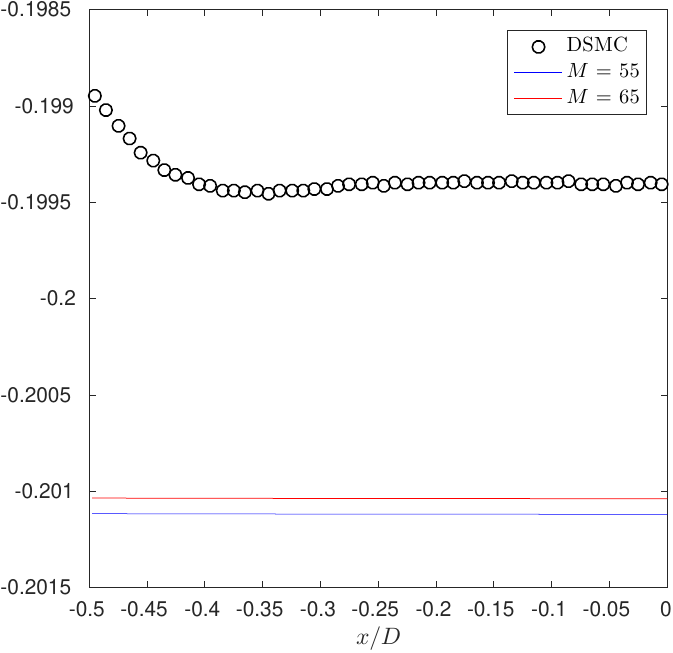}} \hfill
  \subfloat[Heat flux, $q_1~(\rm kg/s^{3})$]{\includegraphics[width=0.47\textwidth,clip]{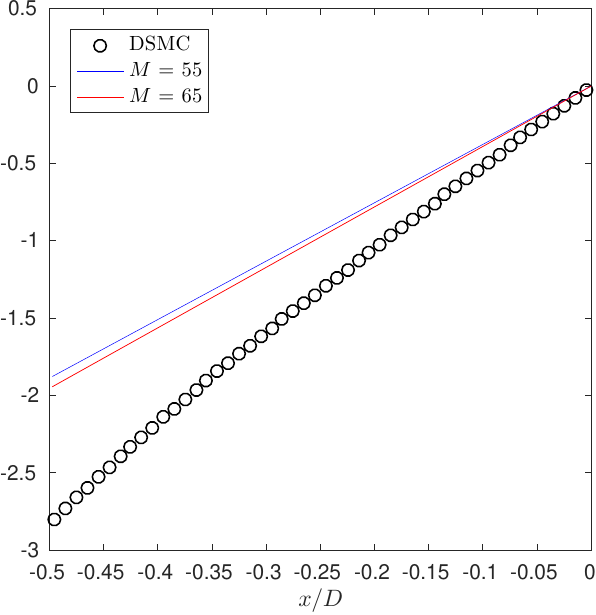}} 
\caption{Solution of the Couette flow for the quadratic collision term
  \eqref{eq:Qstar_f} with $M_{0}=10$ and $D=0.00074 {\rm
    m}$ ($\Kn=12.5$).} 
  \label{fig:couette-Kn125-binaryIM10}
\end{figure}

\subsubsection{The Fourier flow}
\label{sec:num-exp-fourier}

The second benchmark problem is the Fourier flow which also studies
the motion of the gas between two infinite parallel plates. In
contrast to the planar Couette flow, both plates are stationary, while
their temperature is different. Specifically, the left plate has the
temperature $T_{l}^{w} = 273.15{\rm K}$ and the right plate has the
temperature $T_{r}^{w} = 4 T_{l}^{w} = 1092.6{\rm K}$. In this
situation, the gas also reaches a steady state as time goes. To
simulate it, we adopt $\lu = 0$ and
$\lT = \frac{k_{B}}{\mfm}T_{r}^{w}$. The accommodation coefficient in
the boundary condition is set to be $\omega = 1$, and the
computational domain is still $[-D/2, D/2]$ with $D$ being the
distance between two plates. Four distances $D=0.092456 {\rm m}$,
$0.018491 {\rm m}$, $D = 0.003698 {\rm m}$ and $0.00074 {\rm m}$ with
the corresponding dimensionless Knudsen number $\Kn = 0.1$, $0.5$,
$2.5$ and $12.5$ respectively, are considered. Only results for
quadratic collision term \eqref{eq:Qstar_f} are presented.

(1) $D=0.092456{\rm m}$, $\Kn=0.1$: Numerical results for density
$\rho$, temperature $T$, normal stress $\sigma_{11}$ and heat flux
$q_{1}$ with $M_{0}=5$, together with the DSMC solutions, are shown in
\figurename~\ref{fig:fourier-Kn01-binaryIM5}. Our results coincide
very well with the DSMC solutions for density and temperature, while a
small deviation for normal stress $\sigma_{11}$ and heat flux $q_{1}$
can be observed. Note that for heat flux $q_{1}$, the relative
deviation between the DSMC solution and our solution is less than
$3\%$ for all $M$. It is worth mentioning that $q_1$ should be a
constant in the steady-state solution, while the DSMC method provides
a slanting profile. Such a result suggests the possible numerical
error in the DSMC method, although we have run the DSMC code more than
six days. It is left to the future work to determine what this
constant should be.

Nevertheless, the deviation between our results and the DSMC solutions
can be reduced by increasing $M_{0}$ in the collision term. As an
example, we plot the results of normal stress $\sigma_{11}$ and heat
flux $q_{1}$ with $M_{0}=10$ in
\figurename~\ref{fig:fourier-Kn01-binaryIM10}. Remarkable improvement
can be observed.

(2) $D=0.018491{\rm m}$, $\Kn=0.5$: Numerical results with $M_{0}=5$
and $M_{0}=10$, are presented in
\figurename~\ref{fig:fourier-Kn05-binaryIM5} and
\ref{fig:fourier-Kn05-binaryIM10}, respectively. For this larger
Knudsen number, evident deviations for all plotted quantities, in
comparison to the DSMC results, can be observed even for a large $M$
in the case $M_{0}=5$. This indicates $M_{0}=5$ is not enough for the
simulation in this case.

As shown in \figurename~\ref{fig:fourier-Kn05-binaryIM10}, the results
with $M_{0}=10$ again show considerable improvement, especially for
$M$ which is odd and larger than $20$. For these $M$, all plotted
quantities match the DSMC solutions quite well. It can also be
observed that convergence of all plotted quantities with an even $M$
is much slower, especially in the region near the left plate. The
underlying reason remains to be further studied.

(3) $D = 0.003698 {\rm m}$, $\Kn = 2.5$: Numerical solutions for $M_0
= 10$ in this case are given in
\figurename~\ref{fig:fourier-Kn25-binaryIM10}, which shows the results
for $M = 35, 45, 55, 65$. Despite a large Knudsen number, for all
quantities, the profiles for different $M$ are very close to each
other, and they all agree well with DSMC solutions. The maximum
relative deviation for all these quantities is less than $0.2\%$,
which again shows the applicability of the Hermite spectral method for
transitional flows.

(4) $D = 0.00074 {\rm m}$, $\Kn = 12.5$: Numerical results for density
$\rho$, temperature $T$, normal stress $\sigma_{11}$ and heat flux
$q_{1}$ with $M_{0}=10$, together with the DSMC solutions, are shown
in \figurename~\ref{fig:fourier-Kn125-binaryIM10}. It seems that our
solutions are comparable with the DSMC solution for all these four
quantities with $M = 65$.  The relative deviation between our solution
and the DSMC solution for the density and the temperature is $1.2\%$
and $1.5\%$, respectively. But for the normal stress, the relative
deviation is up to $15\%$. The relative deviation for heat flux is
still quite small as to $0.1\%$. This may indicate the inadequacy of
$M_0$ in this simulation.

\begin{figure}[!htb]
  \centering
  \subfloat[Density, $\rho~({\rm kg\cdot m^{-3}})$]{\includegraphics[width=0.49\textwidth,clip]{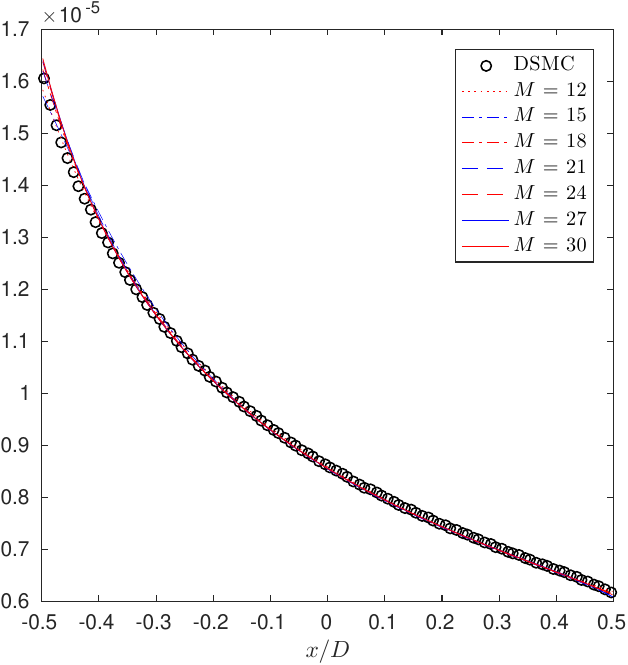}}\hfill
  \subfloat[Temperature, $T~({\rm K})$]{\includegraphics[width=0.50\textwidth,clip]{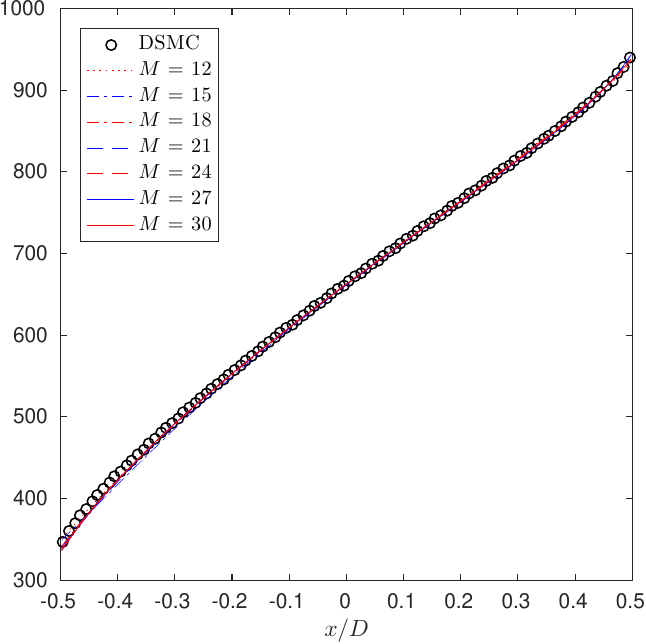}} \\
  \subfloat[Normal stress, $\sigma_{11}~({\rm kg \cdot m^{-1}\cdot s^{-2}})$]{\includegraphics[width=0.48\textwidth,clip]{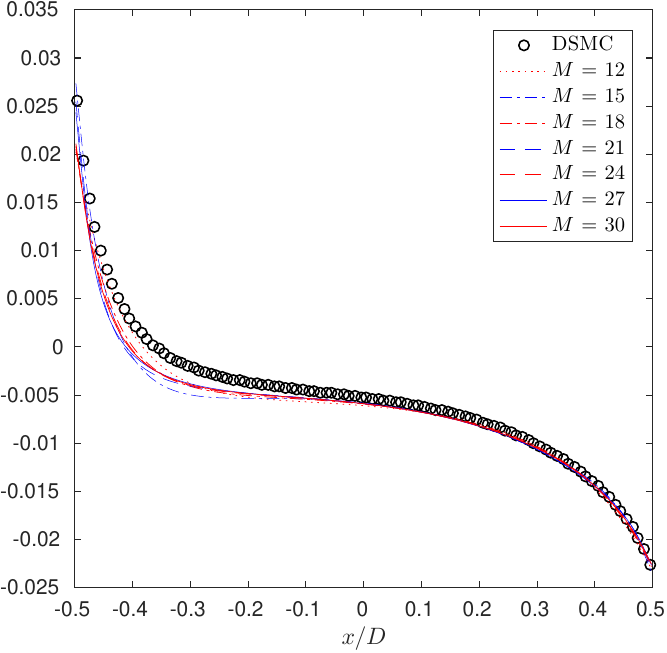}} \hfill
  \subfloat[Heat flux, $q_1~(\rm kg/s^{3})$]{\includegraphics[width=0.5\textwidth,clip]{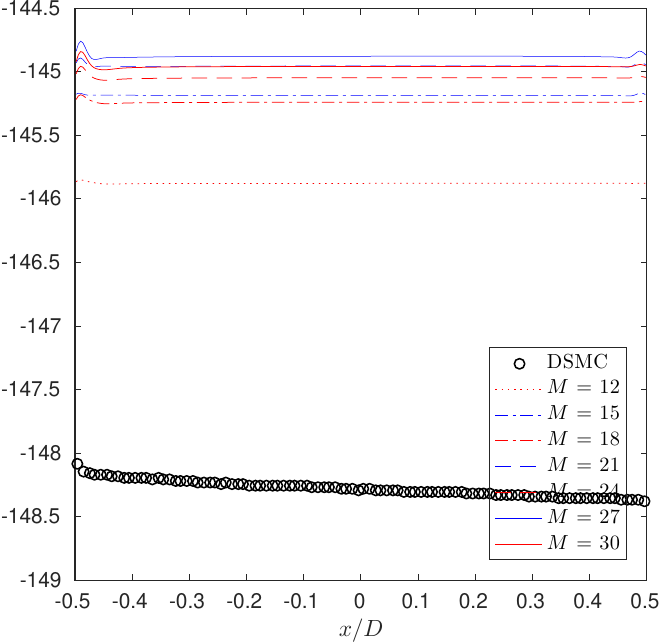}} 
\caption{Solution of the Fourier flow for the quadratic collision term
  \eqref{eq:Qstar_f} with $M_{0}=5$ and $D=0.092456 {\rm
    m}$ ($\Kn=0.1$).} 
  \label{fig:fourier-Kn01-binaryIM5}
\end{figure}

\begin{figure}[!htb]
  \centering
  \subfloat[Normal stress, $\sigma_{11}~({\rm kg \cdot m^{-1}\cdot s^{-2}})$]{\includegraphics[width=0.48\textwidth,clip]{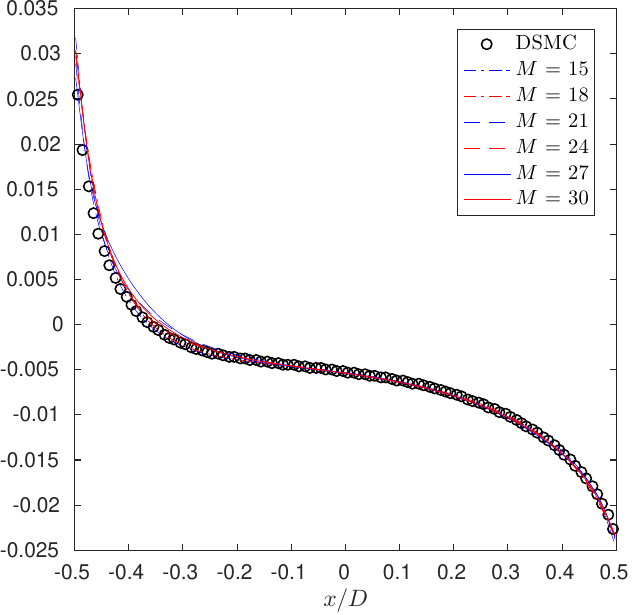}} \hfill
  \subfloat[Heat flux, $q_1~(\rm kg/s^{3})$]{\includegraphics[width=0.5\textwidth,clip]{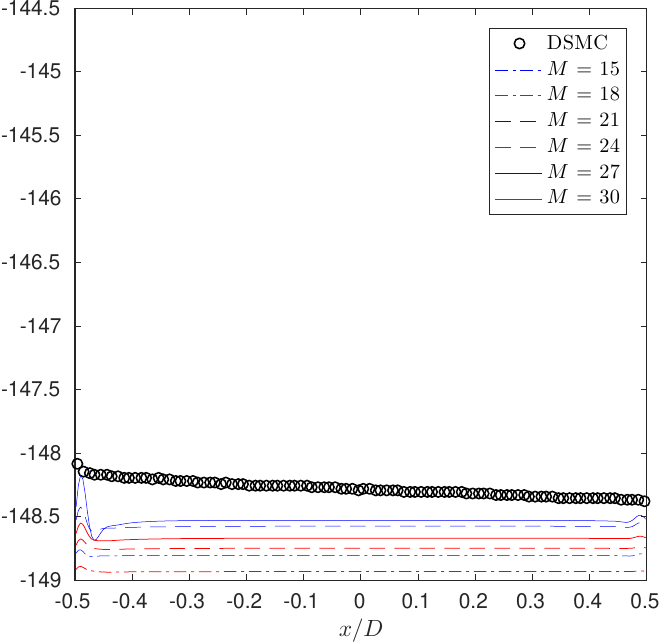}} 
\caption{Normal stress and Heat flux of the Fourier flow for the quadratic collision term
  \eqref{eq:Qstar_f} with $M_{0}=10$ and $D=0.092456 {\rm
    m}$ ($\Kn=0.1$).} 
  \label{fig:fourier-Kn01-binaryIM10}
\end{figure}

\begin{figure}[!htb]
  \centering
  \subfloat[Density, $\rho~({\rm kg\cdot m^{-3}})$]{\includegraphics[width=0.49\textwidth,clip]{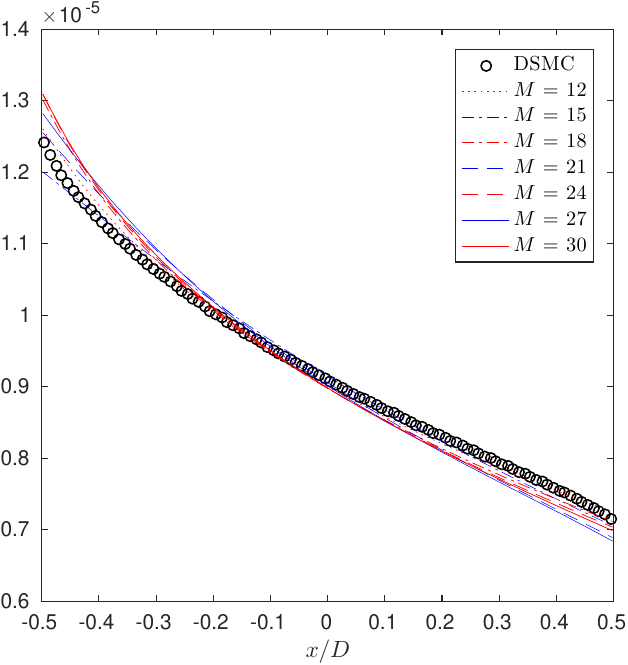}}\hfill
  \subfloat[Temperature, $T~({\rm K})$]{\includegraphics[width=0.49\textwidth,clip]{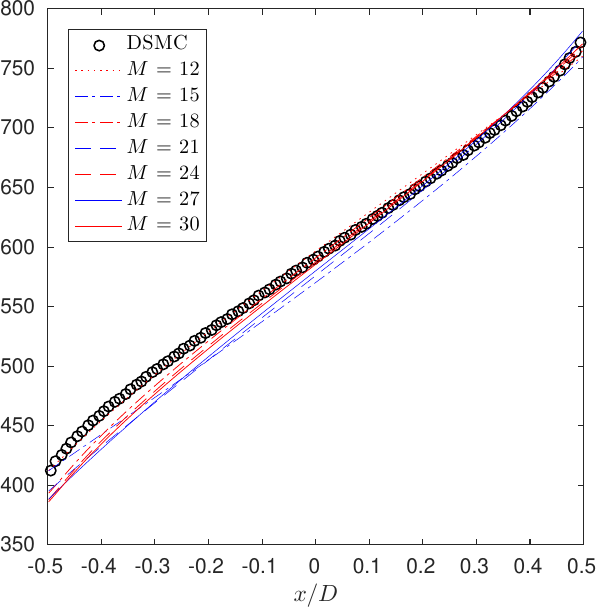}} \\
  \subfloat[Normal stress, $\sigma_{11}~({\rm kg \cdot m^{-1}\cdot s^{-2}})$]{\includegraphics[width=0.48\textwidth,clip]{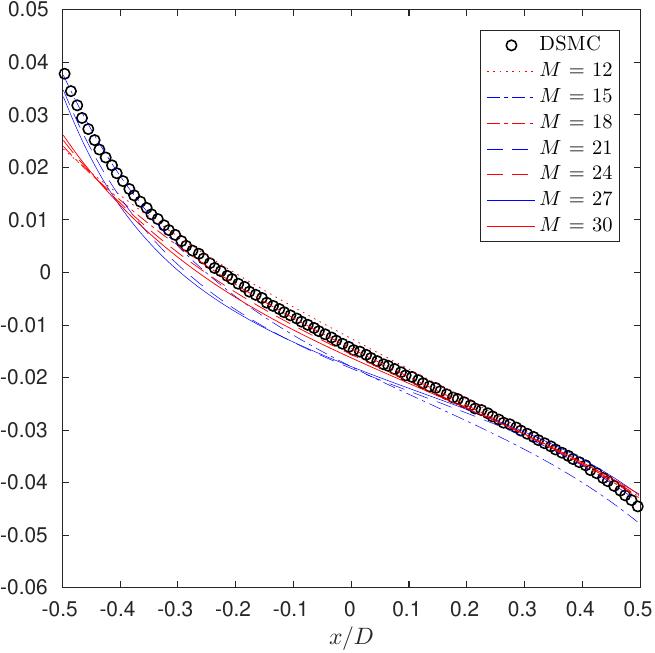}} \hfill
  \subfloat[Heat flux, $q_1~(\rm kg/s^{3})$]{\includegraphics[width=0.49\textwidth,clip]{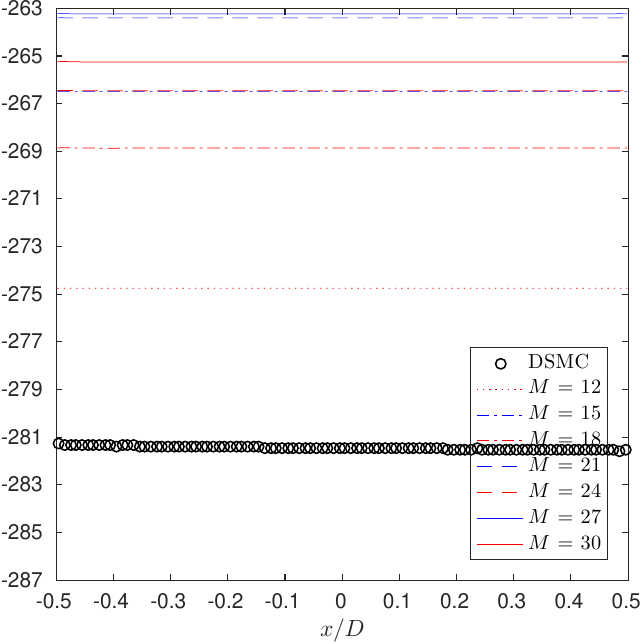}} 
  \caption{Solution of the Fourier flow for the quadratic collision
    term \eqref{eq:Qstar_f} with $M_{0}=5$ and $D=0.018491 {\rm m}$
    ($\Kn=0.5$).} 
  \label{fig:fourier-Kn05-binaryIM5}
\end{figure}

\begin{figure}[!htb]
  \centering
  \subfloat[Density, $\rho~({\rm kg\cdot m^{-3}})$]{\includegraphics[width=0.49\textwidth,clip]{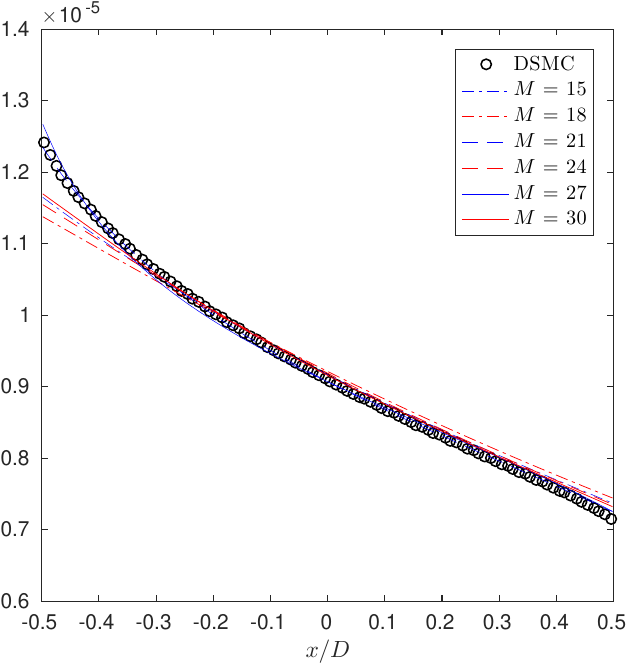}}\hfill
  \subfloat[Temperature, $T~({\rm K})$]{\includegraphics[width=0.49\textwidth,clip]{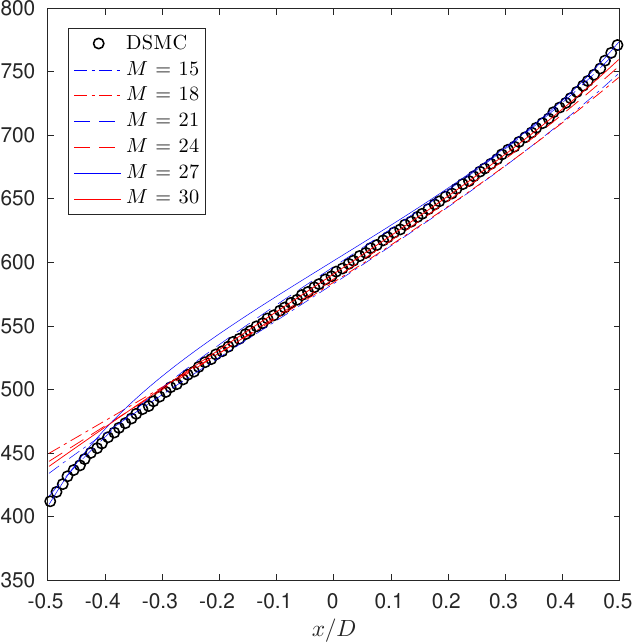}} \\
  \subfloat[Normal stress, $\sigma_{11}~({\rm kg \cdot m^{-1}\cdot s^{-2}})$]{\includegraphics[width=0.48\textwidth,clip]{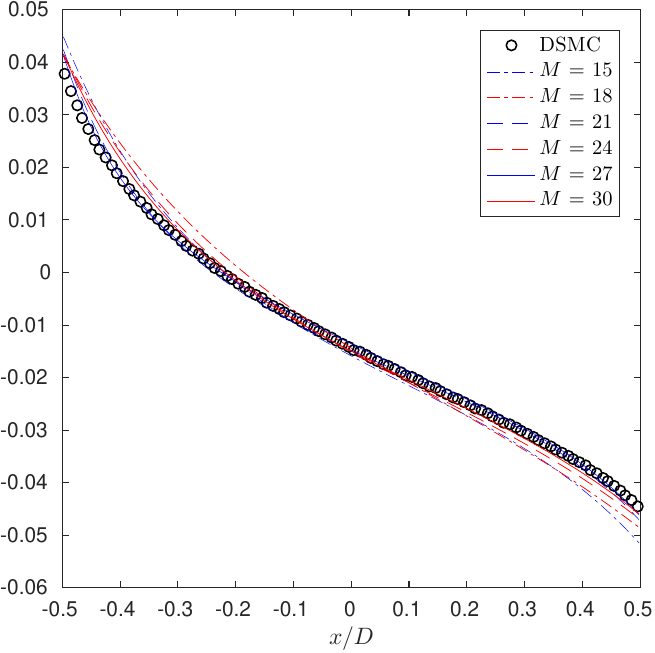}} \hfill
  \subfloat[Heat flux, $q_1~(\rm kg/s^{3})$]{\includegraphics[width=0.49\textwidth,clip]{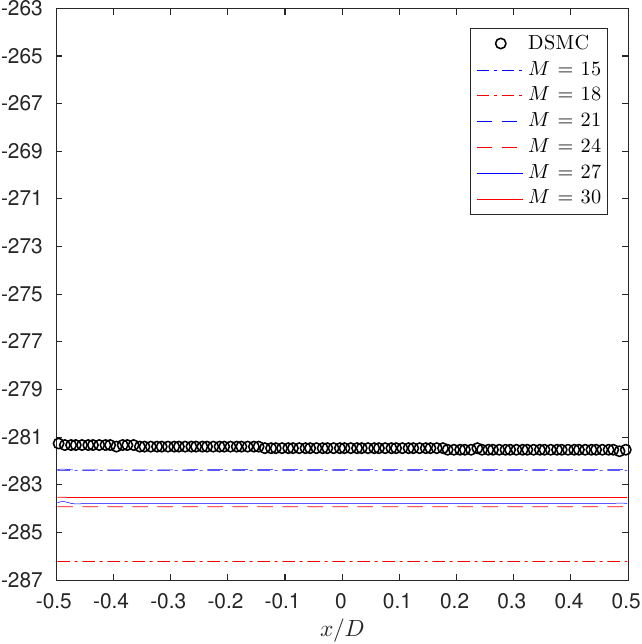}} 
  \caption{Solution of the Fourier flow for the quadratic collision
    term \eqref{eq:Qstar_f} with $M_{0}=10$ and $D=0.018491 {\rm m}$
    ($\Kn=0.5$).} 
  \label{fig:fourier-Kn05-binaryIM10}
\end{figure}

\begin{figure}[!htb]
  \centering
  \subfloat[Density, $\rho~({\rm kg\cdot m^{-3}})$]{\includegraphics[width=0.49\textwidth,clip]{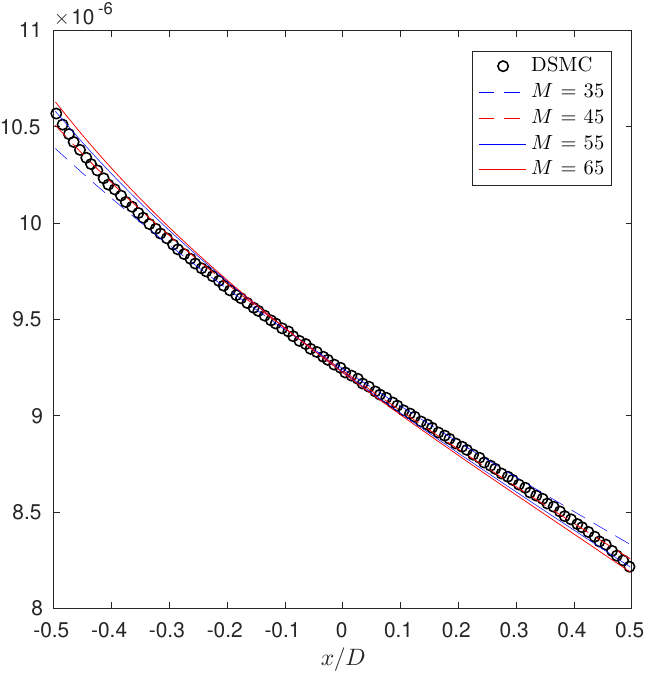}}\hfill
  \subfloat[Temperature, $T~({\rm K})$]{\includegraphics[width=0.49\textwidth,clip]{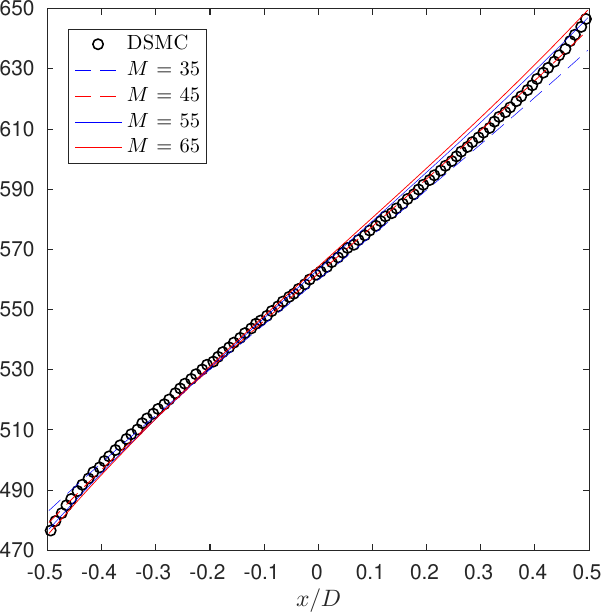}} \\
  \subfloat[Normal stress, $\sigma_{11}~({\rm kg \cdot m^{-1}\cdot s^{-2}})$]{\includegraphics[width=0.48\textwidth,clip]{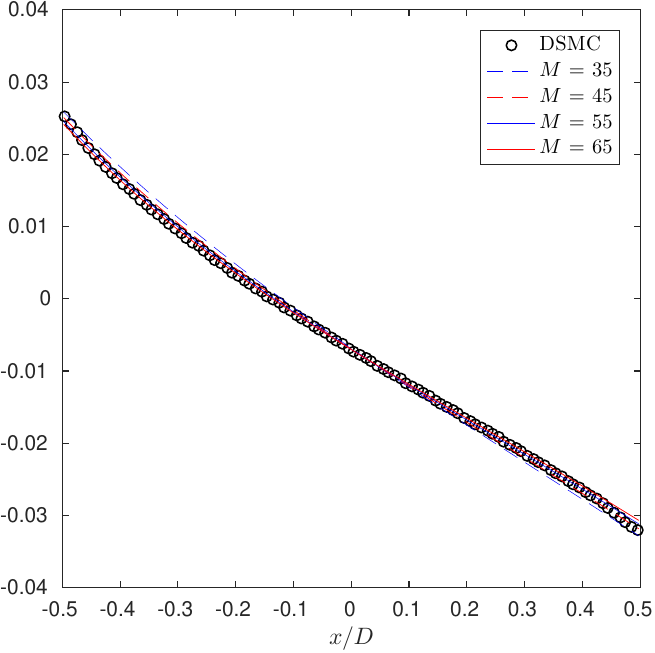}} \hfill
  \subfloat[Heat flux, $q_1~(\rm kg/s^{3})$]{\includegraphics[width=0.49\textwidth,clip]{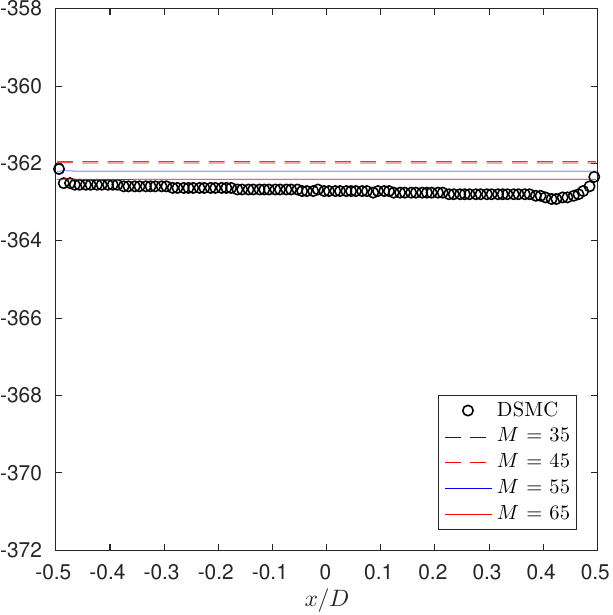}} 
  \caption{Solution of the Fourier flow for the quadratic collision
    term \eqref{eq:Qstar_f} with $M_{0}=10$ and $D=0.003698 {\rm m}$
    ($\Kn=2.5$).} 
  \label{fig:fourier-Kn25-binaryIM10}
\end{figure}

\begin{figure}[!htb]
  \centering
  \subfloat[Density, $\rho~({\rm kg\cdot m^{-3}})$]{\includegraphics[width=0.49\textwidth,clip]{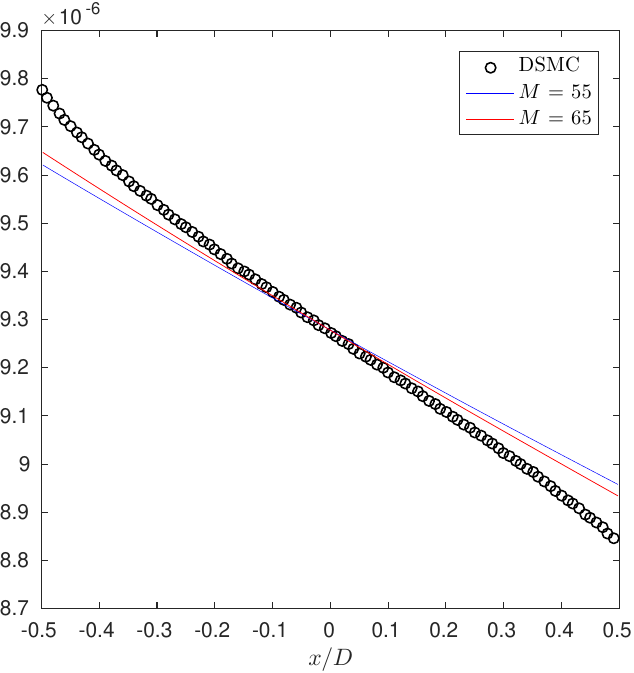}}\hfill
  \subfloat[Temperature, $T~({\rm K})$]{\includegraphics[width=0.49\textwidth,clip]{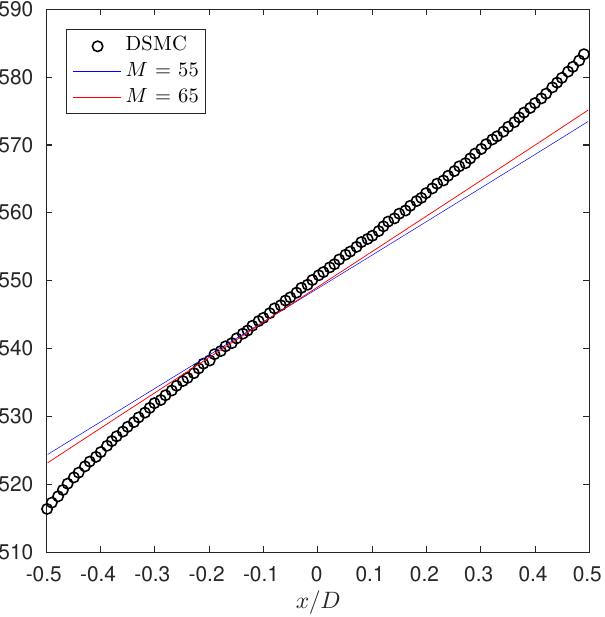}} \\
  \subfloat[Normal stress, $\sigma_{11}~({\rm kg \cdot m^{-1}\cdot s^{-2}})$]{\includegraphics[width=0.48\textwidth,clip]{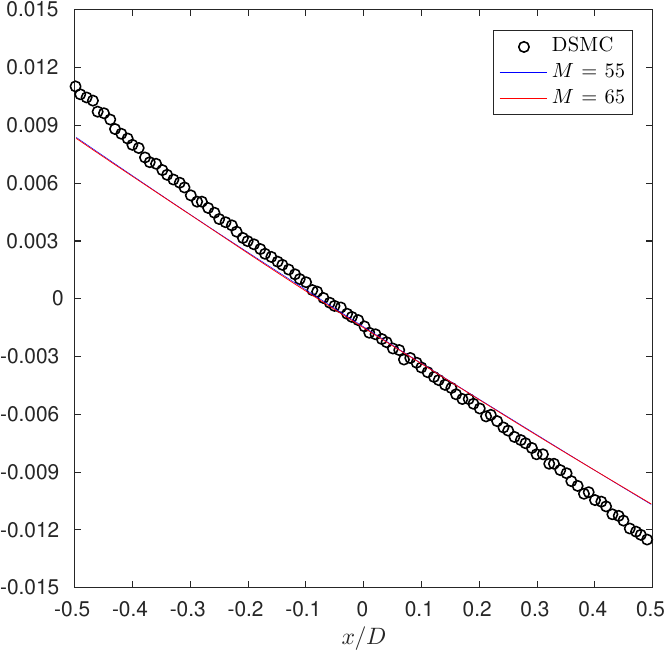}} \hfill
  \subfloat[Heat flux, $q_1~(\rm kg/s^{3})$]{\includegraphics[width=0.49\textwidth,clip]{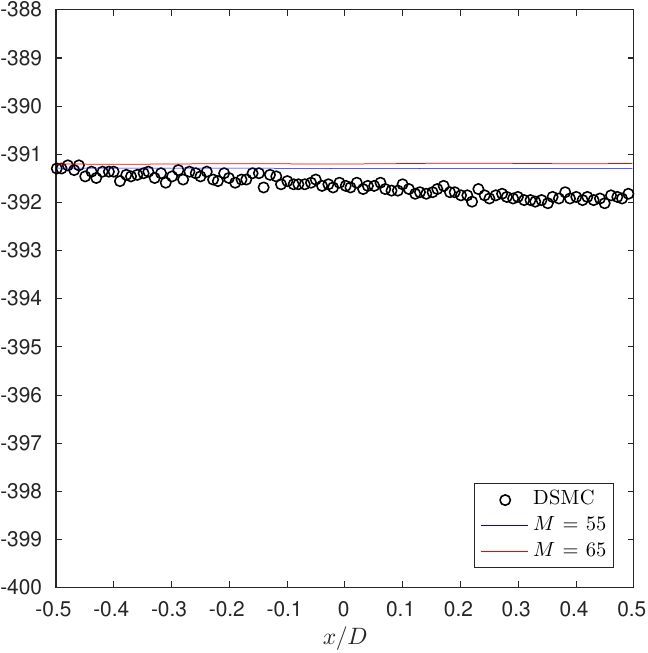}} 
  \caption{Solution of the Fourier flow for the quadratic collision
    term \eqref{eq:Qstar_f} with $M_{0}=10$ and $D=0.00074 {\rm m}$
    ($\Kn=12.5$).} 
  \label{fig:fourier-Kn125-binaryIM10}
\end{figure}

\subsection{Two-dimensional numerical experiments} \label{sec:2D}
As a preliminary study of our method in the multi-dimensional case, we
consider the two-dimensional lid-driven cavity flow which has been
studied in \cite{John2010, Wu2014, Cai2018}. In this case, the argon
gas ($m = 6.63 \times 10^{-26} {\rm kg}$) is confined in a square
cavity with side length $L = 1.25 \times 10^{-6} {\rm m}$.  The
temperature of the cavity walls is $T = T_{\mathrm{ref}} = 273 {\rm
K}$. The viscosity coefficient $\mu$ is set to be
\begin{displaymath}
\mu = \mu_{\mathrm{ref}}
  \left( \frac{T}{T_{\mathrm{ref}}} \right)^{\frac{1}{2} (\eta+3)/(\eta-1)},
\end{displaymath}
where the reference viscosity is $\mu_{\mathrm{ref}} = 2.117 \times
10^{-5} {\rm kg / (m \cdot s)}$, and the value of $\eta$ is $7.45$.
Initially, the gas is in a uniform equilibrium with velocity $\bu = 0
{\rm m/s}$ and temperature $T = 273 {\rm K}$, and the following two
initial densities are considered:
\begin{enumerate}
\item $\rho = 0.891 {\rm kg/m^3}$, corresponding to Knudsen number $\Kn =
  0.1$;
\item $\rho = 0.0891 {\rm kg/m^3}$, corresponding to Knudsen number $\Kn = 1.0$.
\end{enumerate}
The gas flow is driven by the top lid of the cavity, which moves right
at a constant speed $\bu_w = (50, 0, 0) {\rm m/s}$. We expect that the
steady state can be after sufficiently long time. The simulation is
carried out on a $100 \times 100$ grid by explicit time stepping until
$5.24 \times 10^{-8}{\rm s}$. For both cases, we choose $M_0 = 10$,
$\bar{\bu} = 0$ and $\lT = \frac{k_{B}}{m}T_{\rm ref}$ in our
numerical tests.

The simulation is run on the CPU model Intel Xeon E5-2680 v4 @
2.40GHz, and 28 threads are used in the simulation. Details of the
simulations are given in Table \ref{tab:data}. Here the total CPU time
is obtained by the C function {\tt clock()}, whose result is the sum
of CPU time for all threads. Inspired by the tables presented in
\cite{Dimarco2018}, we also provide the CPU time for each time step,
each spatial grid and each degree of freedom for easier comparison.

\begin{table}
\centering
\caption{Run-time data for the lid-driven cavity flow simulations}
\label{tab:data}
\begin{tabular}{ccc}
\hline
Test case & $\Kn = 0.1$ & $\Kn = 1.0$ \\
\hline
$M$ & $25$ & $35$ \\
Number of coefficients ($N_M$) & $3276$ & $8436$ \\
Time step ($\Delta t$) & $2.64 \times 10^{-12}{\rm s}$ & $2.19\times 10^{-12}{\rm s}$ \\
Number of time steps ($N_s$) & $19829$ & $23968$ \\
Total CPU time ($T_{\mathrm{total}}$) & $6.93 \times 10^6 {\rm s}$ & $1.07 \times 10^7 {\rm s}$ \\
CPU time per time step ($T_s = T_{\mathrm{total}} / N_s$) & $3.50 \times 10^2 {\rm s}$ & $4.46 \times 10^2 {\rm s}$ \\
CPU time per grid ($T_g = T_{s} / 100^2$) & $3.50 \times 10^{-2} {\rm s}$ & $4.46 \times 10^{-2} {\rm s}$ \\
CPU time per degree of freedom ($T_d = T_g / N_M$) & $1.07 \times 10^{-5} {\rm s}$ & $5.30 \times 10^{-6} {\rm s}$ \\
\hline
\end{tabular}
\end{table}

The results are again compared with DSMC results \cite{John2010},
which are provided in Figure \ref{fig:cavity-Kn01} and
\ref{fig:cavity-Kn1}. In general, two results agree well with each
other, while some discrepancy can be found on the boundary of the
domain. Such discrepancy is probably related to the Gibbs phenomenon
in the spectral method, since the distribution function on the
boundary of the domain is generally discontinuous. Possible
improvement includes using filters \cite{Aguirre2008} or other
boundary conditions \cite{Sarna2018}, which will be part of our future
work.

\begin{figure}[!htb]
  \centering
  \subfloat[Temperature]{\includegraphics[height=0.37\textwidth,clip]{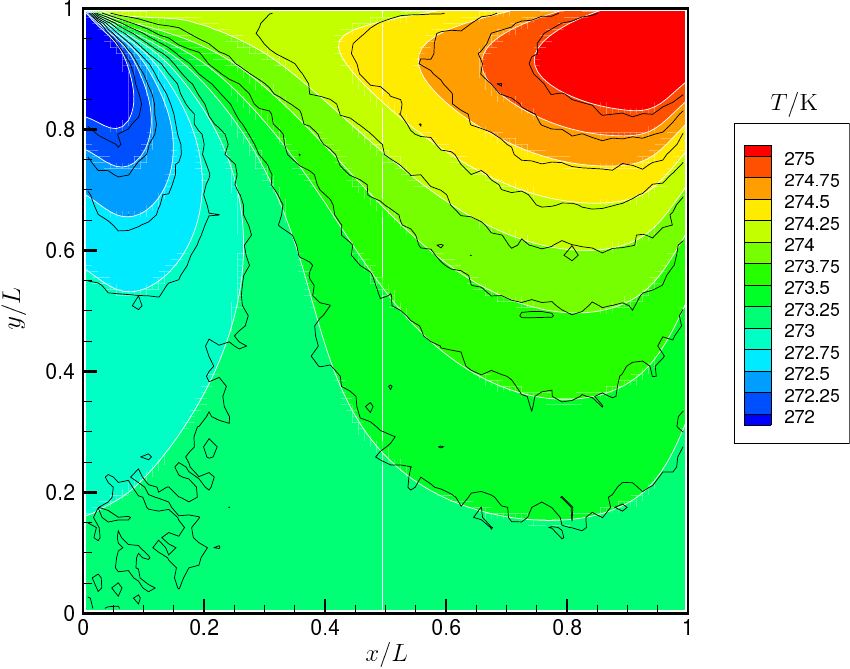}} \hfill
  \subfloat[Shear stress]{\includegraphics[height=0.37\textwidth,clip]{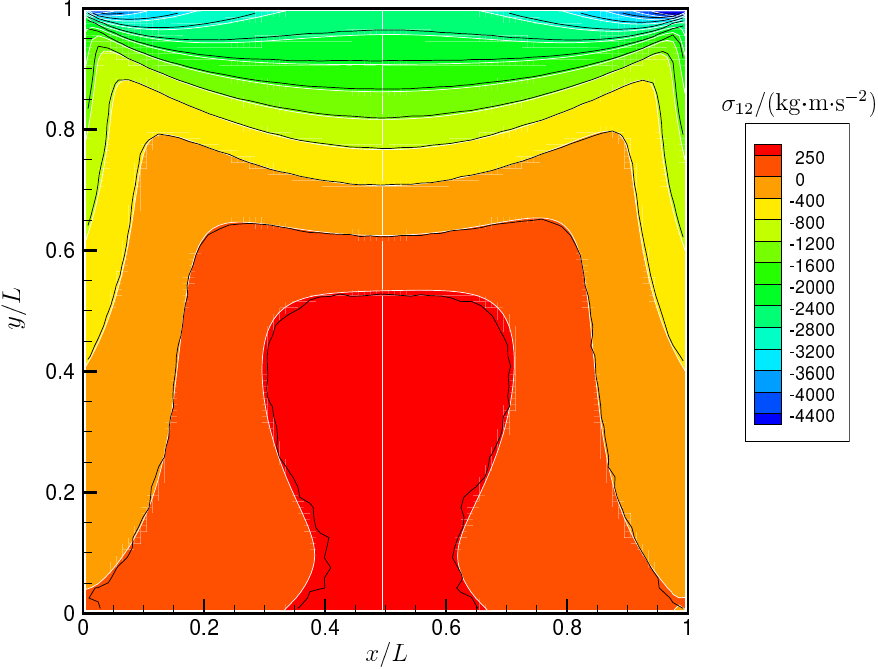}}
  \caption{Numerical results for $\Kn = 0.1$. White contours: Hermite spectral method. Black contours: DSMC.}
  \label{fig:cavity-Kn01}
\end{figure}

\begin{figure}[!htb]
  \centering
  \subfloat[Temperature]{\includegraphics[height=0.37\textwidth,clip]{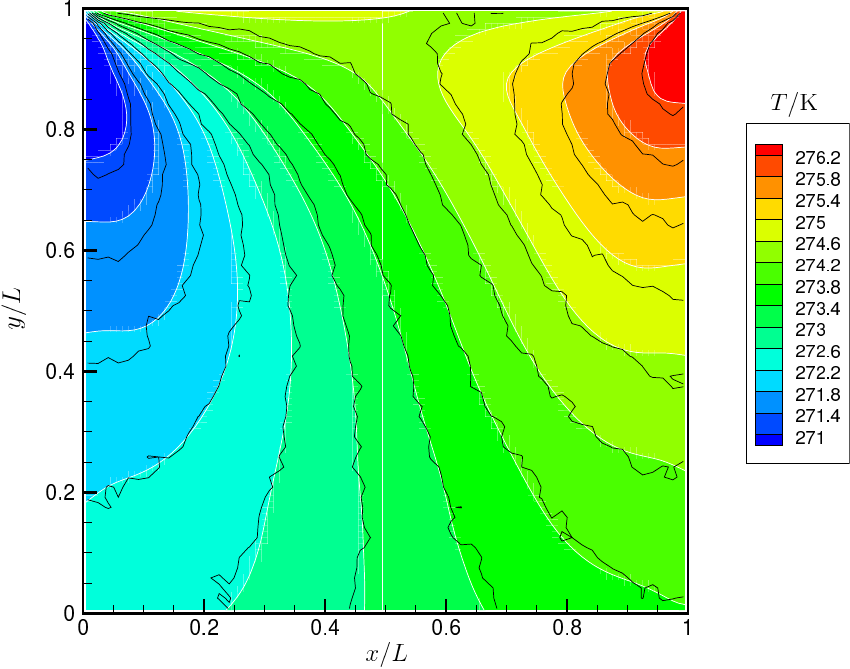}} \hfill
  \subfloat[Shear stress]{\includegraphics[height=0.37\textwidth,clip]{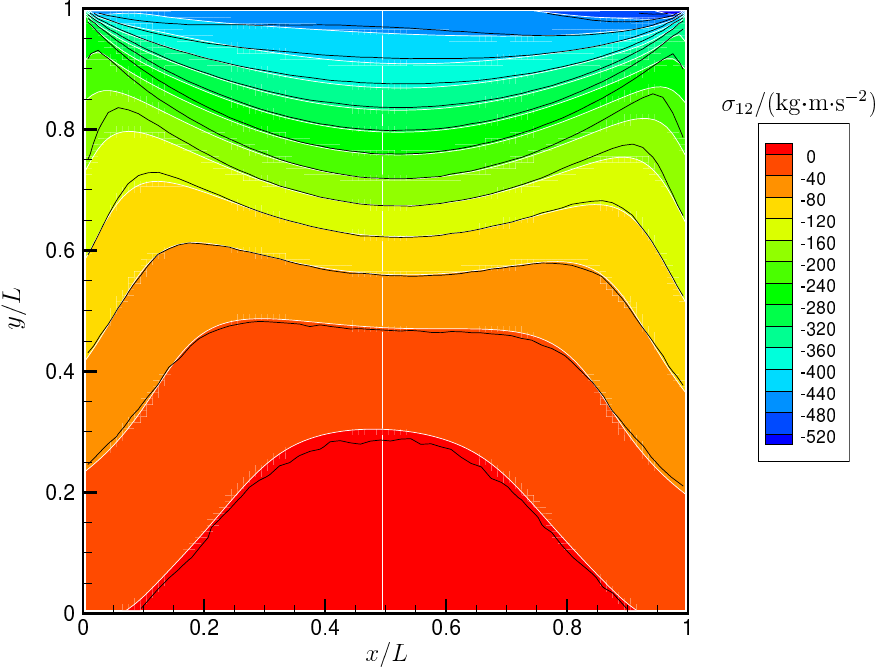}}
  \caption{Numerical results for $\Kn = 1.0$. White contours: Hermite spectral method. Black contours: DSMC.}
  \label{fig:cavity-Kn1}
\end{figure}


\section{Conclusion} \label{sec:conclusion}
Based on the Hermite spectral method, we have developed a numerical solver for
the Boltzmann equation with an approximate collision operator proposed in
\cite{QuadraticCol}. The approximate collision operator is derived from the
original quadratic collision operator, but the quadratic form is preserved only
for the first few moments. Our numerical simulation shows that a small number
of degrees of freedom for the quadratic part can already provide much better
results than the linear models, which makes it possible to design numerical
methods that can well balance the workload and the accuracy. Our major
contribution to the algorithm is a special implementation of the collision
operator. As is mentioned in Section \ref{sec:intro}, the implementation of
such a special collision operator in the spatially inhomogeneous case is not as
straightforward as for the spatially homogeneous and normalized equation
considered in \cite{QuadraticCol}. By introducing a fast algorithm to change
basis functions, we eventually obtain a numerical scheme with time complexity
$O(M_0^9 + M^4)$.

Such a numerical cost makes the algorithm highly promising when applied to more
complicated problems. Research works on more multi-dimensional problems and
polyatomic gases are ongoing.

\newpage

\bibliographystyle{siamplain}
\bibliography{article}

\begin{thebibliography}{10}

\bibitem{Aguirre2008}
{\sc J.~Aguirre and J.~Rivas}, {\em A spectral viscosity method based on
  {H}ermite functions for nonlinear conservation laws}, SIAM J. Numer. Anal.,
  46 (2008), pp.~1060--1078.

\bibitem{Bird1963}
{\sc G.~A. Bird}, {\em Approach to translational equilibrium in a rigid sphere
  gas}, Phys. Fluids, 6 (1963), pp.~1518--1519.

\bibitem{Bird}
{\sc G.~A. Bird}, {\em Molecular Gas Dynamics and the Direct Simulation of Gas
  Flows}, Oxford: Clarendon Press, 1994.

\bibitem{Cai2015Framework}
{\sc Z.~Cai, Y.~Fan, and R.~Li}, {\em A framework on moment model reduction for
  kinetic equation}, SIAM J. Appl. Math., 75 (2015), pp.~2001--2023.

\bibitem{Qiao}
{\sc Z.~Cai, Y.~Fan, R.~Li, and Z.~Qiao}, {\em Dimension-reduced hyperbolic
  moment method for the {B}oltzmann equation with {BGK}-type collision},
  Commun. Comput. Phys., 15 (2014), pp.~1368--1406.

\bibitem{NRxx}
{\sc Z.~Cai and R.~Li}, {\em Numerical regularized moment method of arbitrary
  order for {B}oltzmann-{BGK} equation}, SIAM J. Sci. Comput., 32 (2010),
  pp.~2875--2907.

\bibitem{Cai2012}
{\sc Z.~Cai, R.~Li, and Z.~Qiao}, {\em {NR$xx$} simulation of microflows with
  shakhov model}, SIAM J. Sci. Comput., 34 (2012), pp.~A339--A369.

\bibitem{Microflows1D}
{\sc Z.~Cai, R.~Li, and Z.~Qiao}, {\em Globally hyperbolic regularized moment
  method with applications to microflow simulation}, Computers and Fluids, 81
  (2013), pp.~95--109.

\bibitem{Cai2015}
{\sc Z.~Cai and M.~Torrilhon}, {\em Approximation of the linearized {B}oltzmann
  collision operator for hard-sphere and inverse-power-law models}, J. Comput.
  Phys., 295 (2015), pp.~617--643.

\bibitem{Cai2018}
{\sc Z.~Cai and M.~Torrilhon}, {\em Numerical simulation of microflows using
  moment methods with linearized collision operator}, J. Sci. Comput., 74
  (2018), pp.~336--374.

\bibitem{Cowling}
{\sc S.~Chapman and T.~G. Cowling}, {\em The Mathematical Theory of Non-uniform
  Gases, Third Edition}, Cambridge University Press, 1990.

\bibitem{Chen2015}
{\sc S.~Chen, K.~Xu, and Q.~Cai}, {\em A comparison and unification of
  ellipsoidal statistical and {S}hakhov {BGK} models}, Adv. Appl. Math. Mech.,
  7 (2015), pp.~245--266.

\bibitem{Dimarco2018}
{\sc G.~Dimarco, R.~Loub{\'e}re, J.~Narski, and T.~Rey}, {\em An efficient
  numerical method for solving the {B}oltzmann equation in multidimensions}, J.
  Comput. Phys., 353 (2018), pp.~46--81.

\bibitem{Hu2017}
{\sc I.~M. Gamba, J.~R. Haack, C.~D. Hauck, and J.~Hu}, {\em A fast spectral
  method for the {B}oltzmann collision operator with general collision
  kernels}, SIAM J. Sci. Comput., 39 (2017), pp.~B658--B674.

\bibitem{Gamba2018}
{\sc I.~M. Gamba and S.~Rjasanow}, {\em {G}alerkin-{P}etrov approach for the
  {B}oltzmann equation}, J. Comput. Phys., 366 (2018), pp.~341--365.

\bibitem{Goldstein1989}
{\sc D.~Goldstein, B.~Sturtevant, and J.~E. Broadwell}, {\em Investigations of
  the motion of discrete-velocity gases}, Progress in Astronautics and
  Aeronautics, 117 (1989), pp.~100--117.

\bibitem{Grad}
{\sc H.~Grad}, {\em On the kinetic theory of rarefied gases}, Comm. Pure Appl.
  Math., 2 (1949), pp.~331--407.

\bibitem{HLL}
{\sc A.~Harten, P.~D. Lax, and B.~V. Leer}, {\em On upstream differencing and
  {G}odunov-type schemes for hyperbolic conservation laws}, SIAM Review, 25
  (1983), pp.~35--61.

\bibitem{hu2014nmg}
{\sc Z.~Hu and R.~Li}, {\em A nonlinear multigrid steady-state solver for 1{D}
  microflow}, Computers and Fluids, 103 (2014), pp.~193--203.

\bibitem{hu2016acceleration}
{\sc Z.~Hu, R.~Li, and Z.~Qiao}, {\em Acceleration for microflow simulations of
  high-order moment models by using lower-order model correction}, J. Comput.
  Phys., 327 (2016), pp.~225--244.

\bibitem{hu2015}
{\sc Z.~Hu, R.~Li, and Z.~Qiao}, {\em Extended hydrodynamic models and
  multigrid solver of a silicon diode simulation}, Commun. Comput. Phys., 20
  (2016), pp.~551--582.

\bibitem{John2010}
{\sc B.~John, X.-J. Gu, and D.~R. Emerson}, {\em Investigation of heat and mass
  transfer in a lid-driven cavity under nonequilibrium flow conditions},
  Numerical Heat Transfer, Part B: Fundamentals, 58 (2010), pp.~287--303.

\bibitem{Maxwell}
{\sc J.~C. Maxwell}, {\em On stresses in rarefied gases arising from
  inequalities of temperature}, Proc. R. Soc. Lond., 27 (1878), pp.~304--308.

\bibitem{Mouhot2006}
{\sc C.~Mouhot and L.~Pareschi}, {\em Fast algorithms for computing the
  {B}oltzmann collision operator}, Math. Comp., 75 (2006), pp.~1833--1852.

\bibitem{Panferov2002}
{\sc A.~V. Panferov and A.~G. Heintz}, {\em {A new consistent discrete-velocity
  model for the {B}oltzmann equation}}, Math. Method Appl. Sci., 25 (2002),
  pp.~571--593.

\bibitem{Pareschi1996}
{\sc L.~Pareschi and B.~Perthame}, {\em A fourier spectral method for
  homogeneous {B}oltzmann equations}, Transport Theor. Stat., 25 (1996),
  pp.~369--382.

\bibitem{Sarna2018}
{\sc N.~Sarna and M.~Torrilhon}, {\em On stable wall boundary conditions for
  the {H}ermite discretization of the linearised {B}oltzmann equation}, J.
  Stat. Phys., 170 (2018), pp.~101--126.

\bibitem{LWuComparative2017}
{\sc W.~Su, S.~Lindsay, H.~Liu, and L.~Wu}, {\em Comparative study of the
  discrete velocity and lattice {B}oltzmann methods for rarefied gas flows
  through irregular channels}, Phys. Rev. E, 96 (2017), p.~023309,
  \url{https://doi.org/10.1103/PhysRevE.96.023309},
  \url{https://link.aps.org/doi/10.1103/PhysRevE.96.023309}.

\bibitem{QuadraticCol}
{\sc Y.~Wang and Z.~Cai}, {\em Approximation of the {B}oltzmann collision
  operator based on {H}ermite spectral method}, arXiv:1803.11191,  (2018).
\newblock submitted.

\bibitem{Wu2014}
{\sc L.~Wu, J.~M. Reese, and Y.~Zhang}, {\em Solving the {B}oltzmann equation
  deterministically by the fast spectral method: application to gas
  microflows}, J. Fluid Mech., 746 (2014), pp.~53--84.

\end{thebibliography}
\end{document}